\documentclass[11pt]{amsart}
\usepackage{geometry}
\usepackage{graphicx} 
\usepackage{amsthm}
\usepackage{amsfonts}
\usepackage{tikz}
\usepackage{indentfirst}
\usepackage{amsmath,amssymb}

\newcommand\ki[2]{{\color{blue}#1}{\color{red}#2}}

\newtheorem{thm}{Theorem}[section]
\newtheorem{defi}[thm]{Definition}
\newtheorem{lemma}[thm]{Lemma}
\newtheorem{prop}[thm]{Proposition}
\newtheorem{cor}[thm]{Corollary}
\newtheorem{exa}[thm]{Example}
\newtheorem{rema}[thm]{Remark}

\title{Generalized Goulden-Yong duals and signed minimal factorizations}
\author{Shujian Chen and Kiyoshi Igusa}
\date{}

\begin{document}

\maketitle

\begin{abstract}
    In this paper, we give two combinatorial ways to study signed exceptional sequences. First, we show the equivalence between one-way reflections and relatively projective representations. Secondly, we construct generalized Goulden-Yong duals using reverse Garside element actions and folded chord diagrams. We then give two applications of the generalized Goulden-Yong duals: constructing generalized Pr\"{u}fer codes and counting signed factorizations using the matrix-tree theorem.
\end{abstract}

\tableofcontents

\section*{Introduction}
The bijection between minimal factorizations of Coxeter elements and exceptional sequences of quiver representations was studied by Ingalls and Thomas in \cite{ingalls2009noncrossing}. In this paper, we will investigate a more general bijection between signed exceptional sequences and signed minimal factorizations.

Igusa and Todorov generalized the notion of exceptional sequences to signed exceptional sequences defined using a property of modules in the exceptional sequences called relative projectivity \cite{igusa2017signed}, \cite{igusa26history}. They showed that signed exceptional sequences are bijective to ordered cluster-tilting objects which are known as ordered W-Catalan objects, i.e. enumerated by formula $\frac{(\prod_{i=1}^n(e_i+h+1))n!}{|W|}$

On the other hand, Josuat-Verg\`es and Biane considered a property of reflections in minimal factorizations of Coxeter elements which we call one-wayness in this paper \cite{biane2019noncrossing}\cite{MR3395490}. They obtained a refined enumeration formula for a $q$-statistic on one-way reflections in minimal factorizations of a Coxeter element. When $q=2$, the enumeration formula is also the ordered $W$-Catalan number. We refer to the $q=2$ case as signed minimal factorizations.

Therefore, we can see that relative projectivity of modules in exceptional sequences is related to one-wayness of reflections in minimal factorizations of Coxeter elements. We prove this bijection in this paper.

\begin{thm}[Theorem \ref{thm A}]
    Given a Coxeter element c and a minimal factorization $(r_1,\dots,r_n)$ where $r_i \in T$ and its corresponding exceptional sequences $(E_1,\dots,E_n)$.

    Then $r_i$ is one-way if and only if $E_i$ is relatively projective.
\end{thm}

The other focus of this paper is on generalized Goulden-Yong duals.  Goulden and Yong constructed a bijection between the set of minimal factorizations in the symmetric group $S_n$ and the set of trees on $n$ vertices in \cite{MR1897927}. The bijection is constructed by taking the planar dual of the chord diagram and pushing out edge labels from a vertex in the planar dual. 

In order to construct Goulden-Yong duals for type B and type D reflection groups, we need a general notion of planar duals. This is done algebraically via a variation of the braid group actions on minimal factorizations.

Braid group actions on exceptional sequences were first considered by Crawley-Boevey and he showed that the action is transitive on the set of complete exceptional sequences for acyclic quivers \cite{MR1265279}. Ringel extended this to all hereditary artin algebras \cite{ringel1994braid}. The braid group action $\sigma_i$ on a minimal factorization $(r_1,\dots,r_n)$ is given by conjugating $r_{i+1}$ by $r_i$ and permuting the resulting reflection with $r_i$ in the factorization.

The Garside element $\Delta$ is a special element in braid groups. It acts on a minimal factorization $(r_1,\dots,r_n)$ of $c$ by conjugating each $r_{i+1}$ by $r_1r_2\cdots r_i$ and reversing the order of the new reflections. In this paper we will focus on the action of the reverse Garside element $\Delta'$ which is given by conjugating each $r_{i+1}$ by $r_1r_2\cdots r_i$ but not permuting the resulting reflections. We call $\Delta'(r_1,\cdots,r_n)$ the reverse Garside dual of $(r_1,\cdots,r_n)$. The reverse Garside element lies in the augmented braid group $B_n\rtimes \mathbb Z/2$ which acts on the disjoint union of minimal factorizations of $c$ and of $c^{-1}$. We prove two results about the reverse Garside element actions in this paper.

First we show that the reverse Garside element action will preserve one-wayness/relative projectivity.

\begin{thm}[Corollary \ref{thm B}]
    Given an exceptional sequence $(E_1,\dots,E_n)$ for $mod\text-\Lambda$ and its reverse Garside dual $(E_1',\dots,E_n')$, an exceptional sequence for $mod\text-\Lambda^{op}$, $E_i$ is relatively projective if and only if $E_i'$ is relatively projective.

    In particular, the reverse Garside element action $\Delta'$ gives a bijection between signed exceptional sequences of opposite categories that preserves relative projectivity. 
\end{thm}

Then we show that the algebraic operation of taking the reverse Garside dual is equivalent to taking the planar dual on the chord diagram corresponding to the minimal factorizations.

\begin{thm}[Theorem \ref{thm C}]
    Given a Coxeter element c and a minimal factorization $(r_1,\dots,r_n)$, two regions $reg(e_i)$ and $reg(e_j)$ are adjacent in the planar dual if and only if there's an $r_k'$ in the reverse Garside dual such that $r_k'(e_i)=e_j$
\end{thm}

Therefore, we can use the reverse Garside dual to construct planar duals in general for any finite reflection group.

The other ingredient for our construction of generalized Goulden-Yong duals is a variation of the chord diagrams for type B and type D, called folded chord diagrams. The folded version serves our purpose better than the usual chord diagrams for type B and type D because it shares some properties with the type A chord diagrams. In particular, type B and type D folded chord diagrams have $n$ points on the circles and the edge labels from each vertex can be drawn clockwise decreasing to form noncrossing graphs within the circles.

With these two ingredients, we are able to construct the type B and type D Goulden-Yong duals from the planar duals in the folded chord diagrams. Similar to Goulden-Yong duals, they also provide a bijection between minimal factorizations and graphs with certain properties.

\begin{thm}[Theorem \ref{thm D}]
    Type B Goulden-Yong dual gives a bijection between the set of minimal factorizations of a Coxeter element in $B_n$ and the set $BT_n$, the set of rooted trees on n vertices with a loop attached to one of the vertices, such that a reflection $r_i$ in the factorization corresponds to a vertex $i$ in the rooted tree.
\end{thm}

\begin{thm}[Theorem \ref{thm E}]
    Type D Goulden-Yong duals give a 2-to-1 correspondence between the set of minimal factorizations of a Coxeter element in $D_n$ and the set $DT_n$, the set of rooted simple spanning graphs on $n$ vertices with n edges where the root cannot be the smallest vertex in the cycle, such that a reflection $r_i$ in the factorization corresponds to a vertex $i$ in the rooted tree.
\end{thm}

We also record two applications of the generalized Goulden-Yong duals.

The first application is the construction of generalized Pr\"{u}fer codes for type B and type D. Pr\"{u}fer codes give a bijection between the set of trees on $n$ vertices and a sequence of $n-2$ elements where each element is in $[n]$. This provides a bijective proof for Cayley's Formula on the number of trees on $n$ vertices. Composed with Goulden-Yong dual, we obtain a bijective proof for the number of minimal factorizations in $A_n$.

We have the following version of bijections given by Pr\"{u}fer codes for type $B$ and type $D$.

\begin{thm}[Theorem \ref{thm F}]
    Type $B$ Pr\"{u}fer code gives a bijection between $BT_n$, the set of rooted labeled trees with a loop on n vertices, and a set of sequences $S_B=\{(b_1,\dots,b_n) \ | \ b_i \in [n]\}$
\end{thm}

\begin{thm}[Theorem \ref{thm G}]
    Type $D$ Pr\"{u}fer code gives a bijection between $DT_n$, the set of  rooted simple spanning graphs on $n$ vertices with n edges where the root cannot be the smallest vertex in the cycle, and a set of sequences $S_D=\{(d_1,\dots,d_n) \ | \ d_i \in [n-1]\}$
\end{thm}

On the other hand, using results by Igusa and Sen that characterize one-way reflections in minimal factorization in $A_n$, we are able to realize the number of minimal factorizations in $A_n$ as the determinant of a minor of a graph Laplacian. Moreover, we can realize the number of signed minimal factorizations in $A_n$ as the determinant of a minor of a weighted graph Laplacian.

\begin{cor}[Corollary \ref{thm H}]
    For the weighted Laplacian $L'$ of $G_{A_n}$, $det(L'_{n+1})=n!C_{n+1}$ where $C_n$ is the n-th Catalan number.
\end{cor}

The paper is organized as follows. In Section 1, we will cover some basic theories of reflection groups and quiver representations as well as proving the equivalence between relative projectivity and one-wayness. In Section 2, we will review the constructions of chord diagrams and their planar duals. In Section 3, we will define the augmented braid group action and prove some results about the reverse Garside element action. In Section 4, we will construct the generalized Goulden-Yong duals and prove bijections with certain graphs. In Section 5, we will talk about two applications of generalized Goulden-Yong duals.

\section{Reflection groups and relative projectivity}
In this section, we will review the relative projectivity in exceptional sequences and one-wayness in minimal factorizations. We will prove a bijection between these two properties.
\subsection{Reflection groups}
In this section, we will review some of the basic theory of reflection groups

\begin{defi}
    A \textbf{reflection group} $W$ of rank $n$ is a finite subgroup of $GL(\mathbb{R}^n)$ generated by Euclidean reflections $t \in T$ where $T$ is the set of Euclidean reflections. We choose a minimal set of generators called \textbf{simple reflections} and denoted $s_1,\cdots,s_n$.
\end{defi}

For each reflection group, we will pick a Coxeter element in the group and consider its minimal factorizations.

\begin{defi}
    For a reflection group $W$ of rank $n$ and a set of $n$ simple reflections $S \subset T$, \textbf{a Coxeter element $c$} is a product of all generators $s \in S$ in any fixed order.
\end{defi}

\begin{defi}
    Given a Coxeter element $c$, a sequence $(r_1,\dots,r_n)$ of reflections $r_i\in T$ is a \textbf{minimal factorization of $c$} if $c=r_1\dots r_n$ and $n$ is the smallest number for such sequences.
\end{defi}

\subsection{Bijections between minimal factorizations and exceptional sequences}

In this section, we will review some bijections between reflection groups and quiver representations.

We start with the most fundamental bijection which is Gabriel's Theorem.

\begin{thm}
    The positive roots of the root system of a Dynkin diagram are in one-to-one correspondence with the indecomposable representations of a quiver on the Dynkin diagram.
\end{thm}

The bijection is given by sending a positive root $\alpha$ to the indecomposable module $M_\alpha$ with dimension vector $\alpha$. E.g., the simple roots $\alpha_i$ correspond to the simple modules $M_{\alpha_i}=S_i$.

The following theorem by Shi establishes a relation between Coxeter elements and acyclic digraphs.

\begin{thm}[\cite{article}]
    Given a Coxeter system $(W,S,\Gamma)$ where $W$ is a Coxeter group and $S$ is a set of simple reflections in $W$ and $\Gamma$ is the corresponding Coxeter graph.

    There exists a bijection between the set $C(W)$ of Coxeter elements of $W$ and the set $C(\Gamma)$ of acyclic orientations of $\Gamma$
\end{thm}

Finally, the following theorem by Ingalls and Thomas shows the correspondence between minimal factorizations in reflection groups and exceptional sequences in the corresponding quiver representations.

\begin{defi}
    A quiver representation $E$ is \textbf{exceptional} if $Ext(E_i,E_i)=0$ and $End(E_i)$ is a division algebra.
    An \textbf{exceptional sequence} is a sequence of quiver representations $(E_1,\dots,E_n)$ where $E_i$ is exceptional for each i and $Hom(E_j,E_i)=Ext(E_j,E_i)=0$ for $i<j$.
\end{defi}

\begin{thm}[\cite{ingalls2009noncrossing}]
    There's a bijection between the set of complete exceptional sequences in mod-$\Lambda$ of a given quiver and the set of minimal factorizations of the corresponding Coxeter element.
\end{thm}

The bijection sends a complete exceptional sequence $(M_{\beta_1},M_{\beta_2},\cdots,M_{\beta_n})$ to the composition of reflections $s_{\beta_1}s_{\beta_2}\cdots s_{\beta_n}$ where $s_\beta$ is reflection through $\beta$:
\[
	s_\beta(v)=v-2\frac{(v,\beta)}{(\beta,\beta)}
\]
with $(v,w)$ being the standard inner product in $\mathbb R$ (dot product).

\subsection{Bijection between one-wayness and relative projectivity}
In this section, we will prove the bijection between one-way reflections in minimal factorizations and relatively projective representations in exceptional sequences.

In order to define one-wayness of reflections, we first need to define the length function and the absolute length function.

\begin{defi}
    Given any $w\in W$ and a set of simple reflections $S$, the \textbf{length} $l(w)$ of $w$ is the smallest $k$ such that $w=s_1\dots s_k$ where $s_i \in S$.
\end{defi}

\begin{defi}
    Given any $w\in W$, the \textbf{absolute length} $l_T(w)$ of $w$ is the smallest $k$ such that $w=r_1\dots r_k$ where $r_i \in T$.
\end{defi}

We can now define one-wayness of a reflection.

\begin{defi}
Let $(W, S)$ be a Coxeter system. The \textbf{Bruhat order}, denoted by $\le_B$, is the partial order on $W$ defined by the subword property: for $v, w \in W$, we have $v \le_B w$ if a reduced expression for $v$ can be obtained as a subword of a reduced expression for $w$.
\end{defi}

\begin{defi}
    We denote the absolute order by $<_a$:

    $v <_a w$ if $l_T(w)=l_T(v^{-1}w)+l_T(v)$ where $l_T$ is the absolute length function.
\end{defi}

\begin{defi} [Definition 4 in \cite{biane2019noncrossing}]
    Given a cover relation $v <_a w$. We say that the cover relation $v <_a w$ is \textbf{one-way} if $v <_B w$ and it's $\textbf{two-way}$ if $v >_B w$
\end{defi}

We now switch gears to quiver representations and define relative projectivity. 

\begin{defi}
    Given an exceptional sequence $(E_1,\dots,E_n)$, $E_i$ is \textbf{relatively projective} if $E_i$ is projective in the perpendicular category which is the full subcategory of mod-$\Lambda$ of objects $X$ so that $Hom(E_k,X)=0=Ext(E_k,X)$ for all $k > i$.
\end{defi}

In order to prove the bijection between one-wayness and relative projectivity, we need the following lemmas which characterize them.

\begin{lemma}
    Given an exceptional sequence $(E_1,\dots,E_n)$. $E_i$ is relatively projective if and only if all but one of the simple objects of $\mathcal C=\{E_{i+1},\dots,E_n\}^\perp$ lies in $\{E_i,\dots,E_n\}^\perp$. 
\end{lemma}

\begin{proof}
    Suppose $E_i$ is relatively projective, then $E_i$ is projective in the perpendicular category $\mathcal C$.
    Then there is a simple object $S_i$ in $\mathcal C$ such that $Hom(E_i,S_i) \neq 0$, namely, $S_i$ is the top of $E_i$ in $\mathcal C$. If $E_i$ is projective in $\mathcal C$, $S_i$ is the unique top in $\mathcal C$. Since $E_i$ is projective, there's no simple object $S_j$ in $\mathcal C$ such that $Ext(E_i,S_j) \neq 0$.
    So $S_i$ is the only simple object that is in $\mathcal C$ but not in $\{E_i,\dots,E_n\}^\perp$.

    If $E_i$ isn't relatively projective, then $E_i$ is not projective in $\mathcal C$. Then we take $M=\tau_{\mathcal C}E_i$, the Auslander-Reiten translate of $E_i$ in $\mathcal C$ and we get that $Ext(E_i,M) \neq 0$. By Auslander-Reiten duality we get that $Ext(E_i,soc\,M)\neq0$ and $Hom(E_i,soc\,M)=0$. So the tops of $E_i$ and socles of $M$ are disjoint and not in $E_i^\perp$, so more than one of the simple objects of $\mathcal C$ do not lie in $\{E_i,\dots,E_n\}^\perp$.
\end{proof}

By using the dual argument, we can obtain the following lemma for relatively injective objects.

\begin{lemma}
    Given an exceptional sequence $(E_1,\dots,E_n)$. $E_i$ is relatively injective if and only if all but one of the simple objects of $\mathcal C=^\perp\{E_1,\dots,E_i\}$ lies in $^\perp\{E_1,\dots,E_{i+1}\}$. 
\end{lemma}

There are structures in reflection groups similar to the perpendicular categories in quiver representations.

\begin{lemma}[Lemma 1.4.3 in \cite{BESSIS2003647}]
    Given a Coxeter element $c$ and $v\in W$, then $v <_a c$ if and only if $v$ is a Coxeter element of a parabolic subgroup of W, denoted $W_v$.
\end{lemma}

The correspondence between parabolic subgroups $W_{cr_n\dots r_i}$ and wide subcategories ${r_i,\dots,r_n}^\perp$ is given by restricting the bijection from Gabriel's Theorem to the parabolic subgroups and wide subcategories.

One-way reflections are characterized in a similar way. We will use an equivalent definition of the one-wayness introduced by Josuat-Verg\`es.

\begin{defi}[Definition 2.5 in \cite{MR3395490}]
    Let $\pi_1,\pi_2 \in W$, we say $\pi_2$ is an \textbf{interval refinement} of $\pi_1$ if the simple roots of $W_{\pi_2}$ are included in the simple roots of $W_{\pi_1}$
\end{defi}

\begin{rema}
    Positive roots $W_\pi$ are given by restricting the set positive roots of $W$ to $W_{\pi}$. Simple roots of $W_{\pi}$ are given by the positive roots that cannot be written as a positive linear combination of other positive roots and there's a unique set of such elements in $W_\pi$ which we say is the set of simple roots of $W_\pi$
\end{rema}

The equivalence between one-wayness and interval refinement is proved in \cite{MR3395490}.

\begin{prop}[Proposition 3.4 in \cite{biane2019noncrossing}]
    For $\pi_1, \pi_2 \in W$, the covering relation $\pi_1 <_a \pi_2$ is one-way if and only if $\pi_1$ is an interval refinement of $\pi_2$
\end{prop}

We are now ready to prove the bijection between one-wayness and relative projectivity.

\begin{thm}\label{thm A}
    Given a Coxeter element c and a minimal factorization $(r_1,\dots,r_n)$ where $r_i \in T$ and its corresponding exceptional sequences $(E_1,\dots,E_n)$.

    Then $r_i$ is one-way (i.e., $r_1\cdots r_{i-1}<_a r_1\cdots r_i$ is one-way) if and only if $E_i$ is relatively projective.
\end{thm}

\begin{proof}
Let $\mathcal{C}_{i} = \{E_i, \dots, E_n\}^\perp$ and $\mathcal{C}_{i+1} = \{E_{i+1}, \dots, E_n\}^\perp$ denote the perpendicular categories associated with the suffixes of the exceptional sequence. Note that $\mathcal{C}_{i} \subset \mathcal{C}_{i+1}$ as full subcategories.

By Lemma 1.14, $E_i$ is relatively projective if and only if the set of simple objects of $\mathcal{C}_{i}$ is a subset of the simple objects of $\mathcal{C}_{i+1}$. 

By the bijection established by Ingalls and Thomas \cite{ingalls2009noncrossing} (Theorem 1.7), wide subcategories correspond to parabolic subgroups, and the simple objects of a wide subcategory correspond to the simple roots of the parabolic subgroup. 

Let $W_{\pi_i}$ and $W_{\pi_{i+1}}$ be the parabolic subgroups corresponding to $\mathcal{C}_{i}$ and $\mathcal{C}_{i+1}$ respectively. The condition derived from Lemma 1.14 translates to the inclusion of simple roots:
\[
\Delta(W_{\pi_i}) \subseteq \Delta(W_{\pi_{i+1}}),
\]
where $\Delta(W)$ denotes the set of simple roots of the subgroup $W$.

By Definition 1.17, this condition on simple roots is precisely the definition of $W_{\pi_i}$ being an \textbf{interval refinement} of $W_{\pi_{i+1}}$.

Finally, by Proposition 1.19, the relation $W_{\pi_i} < W_{\pi_{i+1}}$ is an interval refinement if and only if the reflection $r_i$ is \textbf{one-way} i.e., $r_1\cdots r_{i-1}<_a r_1\cdots r_i$ is one-way.

Thus, $E_i$ is relatively projective if and only if $r_i$ is one-way.
\end{proof}

Similarly, we have the following equivalence for relatively injective objects.

\begin{prop}
    Given a Coxeter element c and a minimal factorization $(r_1,\dots,r_n)$ and its corresponding exceptional sequences $(E_1,\dots,E_n)$.

    Then $K(r_1\dots r_k) <_a K(r_1\dots r_{k-1})$ is one-way if and only if $E_k$ is relatively injective, where $K(w)=w^{-1}c$ is the Kreweras complement of $w$.
\end{prop}

\subsection{Bijection between signed minimal Coxeter 
factorizations and ordered cluster-tilting objects}

In this section, we will recall the bijection between signed exceptional sequences and ordered cluster-tilting objects and define signed minimal factorizations using bijection between relative projectivity and one-wayness.

First, we need to define signed exceptional sequences.

\begin{defi}[Definition 2.1 in \cite{igusa2017signed}]
    Given an exceptional sequence $(E_1,\dots,E_n)$ in a wide subcategory $A$, \textbf{a signed exceptional sequence} is an exceptional sequence $(E_1,\dots,E_n)$ with a sequence of signs $(\epsilon_1,\dots,\epsilon_n)$ where $\epsilon_i \in \{1,-1\}$ if $E_i$ is relatively projective and $\epsilon_i=1$ if $E_i$ isn't relatively projective.
\end{defi}

The following theorem by Igusa and Todorov shows the importance of studying signed exceptional sequences.

\begin{thm}[Theorem 2.3 in \cite{igusa2017signed}]
    There's a bijection between the set of ordered cluster tilting sets and the set of signed exceptional sequences.
\end{thm}

In particular, cluster tilting sets are $W$-Catalan objects which means that they are enumerated by the $W$-Catalan numbers $=\prod_{i=1}^n \frac{e_i+h+1}{e_i+1}$ where $e_i$ are the exponents of the root system and $h$ is the Coxeter number.

We can define a signed version of minimal factorization using the bijection in the previous section.

\begin{defi}
    A \textbf{signed minimal factorization} is a minimal factorization $(r_1,\dots,r_n)$ with a sequence of signs $(\epsilon_1,\dots,\epsilon_n)$ where $\epsilon_i \in \{1,-1\}$ if $r_i$ is one-way and $\epsilon_i=1$ if $r_i$ is two-way.
\end{defi}

\begin{rema}
    This definition of signed minimal factorizations is a specialization of the q-statistics on minimal Coxeter factorizations by Josuat-Verg\`es. It can be obtained by setting $q=2$ and putting the weights on one-way reflections.
\end{rema}

Since each object in the cluster category corresponds to an almost positive root, there's a bijection between the set of ordered cluster tilting sets and the set of ordered faces in the corresponding cluster complex. So we have the following combinatorial bijection.

\begin{cor}
    There's a bijection between the set of ordered faces of a cluster complex and the set of signed minimal factorizations.
\end{cor}

\section{Chord diagrams and their planar duals}

In this section, we will first recall the construction of chord diagrams associated to minimal factorizations.

\subsection{Chord diagrams of type $A$,$B$,$D$}
For the construction of chord diagrams, we follow \cite{reading2007clusters} by Reading. We will do them for type $A$, $B$, $D$ families of Coxeter groups.

\subsubsection{Type A chord diagrams} 

A Coxeter element $c$ of $A_n$ is realized as permutation whose cycle type is $(k_1 \dots k_{n+1})$ where $k_i \in [n+1]$ and $k_i \neq k_j$ 

Given a minimal factorization $(r_1,\dots,r_n)$. The circle chord diagram of type $A$ is a circle where $k_i$ are put on the circle clockwise. 

Draw a chord between $k_i$ and $k_j$ if there's a reflection $r_k$ in the factorization such that $r_k(k_i)=k_j$.

Denote the chord between $i$ and $j$ to be $C_{i,j}$

\begin{exa}
    Figure 1 is the chord diagram for the factorization $(e_2-e_1,e_4-e_2,e_3-e_2)$ in $A_3$. $e_2-e_1$ permutes 1 and 2 so it corresponds to the chord $C_{1,2}$. $e_4-e_2$ permutes 4 and 2 so it corresponds to the chord $C_{2,4}$, $e_3-e_2$ permutes 3 and 2 so it corresponds to the chord $C_{2,3}$.
\end{exa}

\begin{figure}
    \centering
    \begin{tikzpicture}
    \draw (0,0) circle (2cm);
  
    \node at (0,2.5) {1};
    \node at (2.5,0) {2};
    \node at (0,-2.5) {3};
    \node at (-2.5,0) {4};
  
    \draw (2,0) -- (0,2);
    \draw (2,0) -- (0,-2);
    \draw (2,0) -- (-2,0);
    \end{tikzpicture}
    \caption{Chord diagram for $(e_2-e_1,e_4-e_2,e_3-e_2)$}
\end{figure}
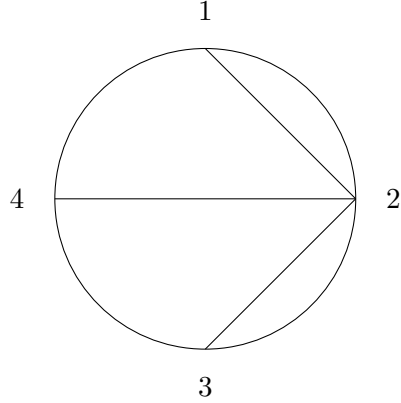

\subsubsection{Type B chord diagrams}

A Coxeter element $c$ of $B_n$ is realized as a signed permutation with a single cycle of length 2n of the form $(k_1 \dots k_{2n})$ where $k_i \in [n] \cup \{-[n]\}$ and $k_i \neq k_j$.

Given a minimal factorization $(r_1,\dots,r_n)$. Circle chord diagram of type $B$ is a circle where $k_i$ are put on the circle clockwise. 

Draw a chord between $k_i$ and $k_j$ if there's a reflection $r_k$ in the factorization such that $r_k(k_i)=k_j$. 

\begin{exa}
    Figure 2 is the chord diagram for the factorization $(e_2, e_2-e_1)$ in $B_2$. $e_2$ permutes 2 and -2 so it corresponds to the chord $C_{2,-2}$. $e_2-e_1$ permutes 1 and 2 as well as  -1 and -2, so it corresponds to the two chords $C_{1,2}$ and $C_{-1,-2}$.
\end{exa}

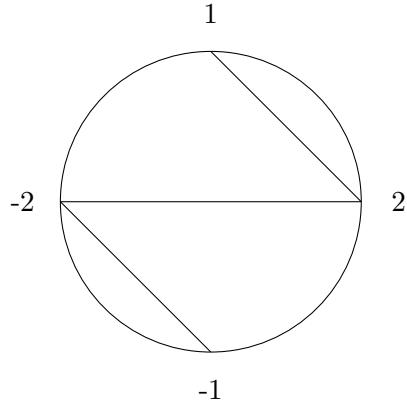
\begin{figure}
    \centering
    \begin{tikzpicture}
    \draw (0,0) circle (2cm);
  
    \node at (0,2.5) {1};
    \node at (2.5,0) {2};
    \node at (0,-2.5) {-1};
    \node at (-2.5,0) {-2};
  
    \draw (2,0) -- (0,2);
    \draw (-2,0) -- (0,-2);
    \draw (2,0) -- (-2,0);
    \end{tikzpicture}
    \caption{Chord diagram for $(e_2,e_2-e_1)$}
\end{figure}

\subsubsection{Type D chord diagrams}

A Coxeter element $c$ of $D_n$ have two cycles, one of length 2 and the other of length $(2n-2)$. The cycle of length 2 is $(i, -i)$ for some $i \in [n]$ 

Without loss of generality, assume the cycle of length 2 is $(1,-1)$. The cycle of length $(2n-2)$ is $(k_1 \dots k_{2n-2})$ where $k_i \in \{\pm2, \pm3, \dots, \pm n\}$

Given a minimal factorization $(r_1,\dots,r_n)$. Circle chord diagram of type $D$ contains two circles where one circle contains the other. $k_i$ are put on the outside circle clockwise and $\{1,-1\}$ are put on the inside counterclockwise.

For $a,b \in \{1, -1\} \cup \{k_1,\dots,k_{2n-2}\}$. Draw a chord between $a$ and $b$ if there's a reflection $r_k$ in the factorization such that $r_k(a)=b$. 

\begin{exa}
    Figure 3 is the chord diagram for the factorization $(e_2+e_1,e_2-e_1,e_3-e_2,e_4-e_3)$ in $D_4$. $e_2+e_1$ permutes 1 and -2 as well as -1 and 2, so it corresponds to the chords $C_{1,-2}$ and $C_{-1,2}$. $e_2-e_1$ permutes 1 and 2 as well as -1 and -2, so it corresponds to the two chords $C_{1,2}$ and $C_{-1,-2}$. $e_3-e_2$ permutes 2 and 3 as well as -2 and -3, so it corresponds to the two chords $C_{2,3}$ and $C_{-2,-3}$. $e_4-e_3$ permutes 3 and 4 as well as -3 and -4, so it corresponds to the two chords $C_{3,4}$ and $C_{-3,-4}$.
\end{exa}

\begin{figure}
    \centering
    \begin{tikzpicture}
    \draw (0,0) circle (2cm);
    \draw (0,0) circle (0.5cm);
  
    \node at (0,2.5) {2};
    \node at (2,1) {3};
    \node at (2,-1) {4};
    \node at (0,-2.5) {-2};
    \node at (-2,-1) {-3};
    \node at (-2,1) {-4};
    \node at (-0.8,0) {1};
    \node at (0.8,0) {-1};
    
    \draw (0,2) -- (1.8,0.8);
    \draw (1.8,0.8) -- (1.8,-0.8);
    \draw (0,-2) -- (-1.8,-0.8);
    \draw (-1.8, -0.8) -- (-1.8, 0.8);
    \draw (0,2) -- (-0.5,0);
    \draw (0,2) -- (0.5,0);
    \draw (0,-2) -- (-0.5,0);
    \draw (0,-2) -- (0.5,0);
    \end{tikzpicture}
    \caption{Chord diagram for $(e_2+e_1,e_2-e_1,e_3-e_2,e_4-e_3)$}
\end{figure}
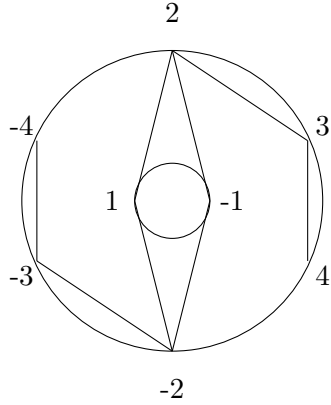

\subsection{Planar dual}
\begin{defi}
    Given a chord diagram. Let $D$ denote the planar dual defined as follows:
    \begin{itemize}
        \item vertices are regions of the chord diagram
        \item two vertices are adjacent if two regions share an edge in the chord diagram
    \end{itemize}
\end{defi}

\begin{exa}
    Figure 4,5,6 demonstrates the planar duals for the previous examples. Since each region contains a unique arc on the circle, we pick a point on the arc as the representative of the corresponding region. Two points are connected if their corresponding region share an edge.
\end{exa}

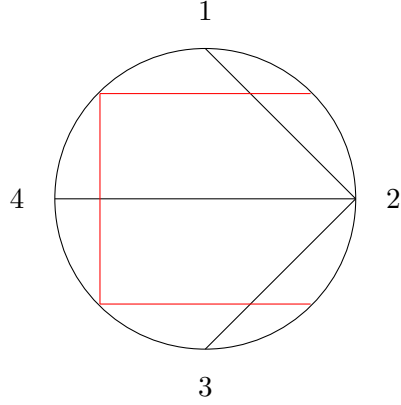
\begin{figure}
    \centering
    \begin{tikzpicture}
    \draw (0,0) circle (2cm);
  
    \node at (0,2.5) {1};
    \node at (2.5,0) {2};
    \node at (0,-2.5) {3};
    \node at (-2.5,0) {4};
  
    \draw (2,0) -- (0,2);
    \draw (2,0) -- (0,-2);
    \draw (2,0) -- (-2,0);
    \draw[red] (1.4, 1.4) -- (-1.4,1.4);
    \draw[red] (-1.4, 1.4) -- (-1.4, -1.4);
    \draw[red] (-1.4, -1.4) -- (1.4, -1.4);
    \end{tikzpicture}
    \caption{Planar dual for $(e_1-e_2,e_2-e_4,e_2-e_3)$ in red}
\end{figure}

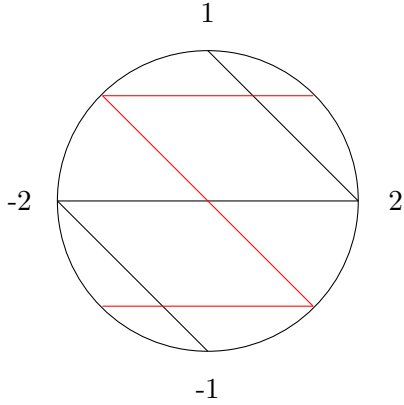
\begin{figure}
    \centering
    \begin{tikzpicture}
    \draw (0,0) circle (2cm);
  
    \node at (0,2.5) {1};
    \node at (2.5,0) {2};
    \node at (0,-2.5) {-1};
    \node at (-2.5,0) {-2};
  
    \draw (2,0) -- (0,2);
    \draw (-2,0) -- (0,-2);
    \draw (2,0) -- (-2,0);
    \draw[red] (-1.4, 1.4) -- (1.4, -1.4);
    \draw[red] (-1.4, 1.4) -- (1.4, 1.4);
    \draw[red] (-1.4, -1.4) -- (1.4, -1.4);
    \end{tikzpicture}
    \caption{Planar dual for $(e_2,e_2-e_1)$ in red}
\end{figure}

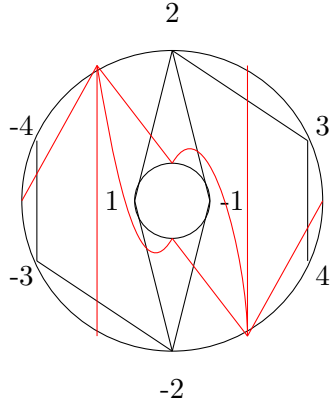
\begin{figure}
    \centering
    \begin{tikzpicture}
    \draw (0,0) circle (2cm);
    \draw (0,0) circle (0.5cm);
  
    \node at (0,2.5) {2};
    \node at (2,1) {3};
    \node at (2,-1) {4};
    \node at (0,-2.5) {-2};
    \node at (-2,-1) {-3};
    \node at (-2,1) {-4};
    \node at (-0.8,0) {1};
    \node at (0.8,0) {-1};
    
    \draw (0,2) -- (1.8,0.8);
    \draw (1.8,0.8) -- (1.8,-0.8);
    \draw (0,-2) -- (-1.8,-0.8);
    \draw (-1.8, -0.8) -- (-1.8, 0.8);
    \draw (0,2) -- (-0.5,0);
    \draw (0,2) -- (0.5,0);
    \draw (0,-2) -- (-0.5,0);
    \draw (0,-2) -- (0.5,0);
    \draw[red] (-1,1.8) -- (-2, 0);
    \draw[red] (-1,1.8) -- (-1,-1.8);
    \draw[red] (1,-1.8) -- (2, 0);
    \draw[red] (1,1.8) -- (1,-1.8);
    \draw[red] (-1, 1.8) -- (0,0.5);
    \draw[red] (0, 0.5) to[out=60,in=90] (1, -1.8);
    \draw[red] (1, -1.8) -- (0,-0.5);
    \draw[red] (0, -0.5) to[out=-120,in=-80] (-1, 1.8);
    \end{tikzpicture}
    \caption{Planar dual for $(e_2+e_1,e_2-e_1,e_3-e_2,e_4-e_3)$ in red}
\end{figure}

\subsection{Edge-colored chord diagram}
\begin{defi}
    Given a minimal factorization $(r_1,\dots,r_n)$. A \textbf{coloring of chord diagram} is a function $f: C \rightarrow \{1,2\}$ where $C$ is the set of chords in a chord diagram such that $f(c)=1$ if the corresponding reflection is two-way and $f(c)=2$ if the corresponding reflection is one-way.
\end{defi}

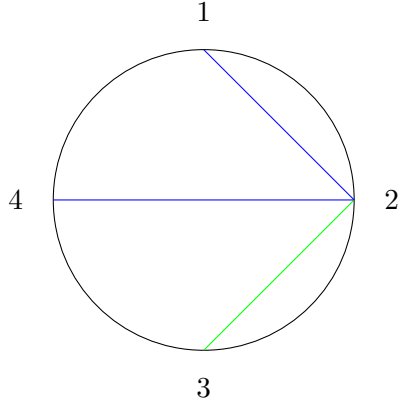
\begin{figure}
    \centering
    \begin{tikzpicture}
    \draw (0,0) circle (2cm);
  
    \node at (0,2.5) {1};
    \node at (2.5,0) {2};
    \node at (0,-2.5) {3};
    \node at (-2.5,0) {4};
  
    \draw[blue] (2,0) -- (0,2);
    \draw[green] (2,0) -- (0,-2);
    \draw[blue] (2,0) -- (-2,0);
    \end{tikzpicture}
    \caption{Edge-colored chord diagram for $(e_1-e_2,e_2-e_4,e_2-e_3)$ where one-way reflections are colored in blue and two-way reflections are colored in green}
\end{figure}

\begin{exa}
    Figure 7 shows the edge-colored chord diagram for $(e_2-e_1,e_4-e_2,e_3-e_2)$. $C_{1,2}$ and $C_{2,4}$ are colored blue since they corresponds to one-way reflections. $C_{2,3}$ is colored green since it corresponds to a two-way reflection.
\end{exa}

\section{Reverse braid group action}
In this section, we will define the reverse braid group actions on minimal factorizations and prove that the reverse Garside duals are equivalent to planar duals in chord diagrams.

\subsection{Reverse braid group action}
In this section we will define the reverse braid group action.

\begin{defi}
    Let $B_m$ be the braid group on $m$ strands and denote its generators by $\sigma_1,\dots,\sigma_{m-1}$. Then \textbf{the braid group map} of $B_m$ on the set of m-element sequences of a finite Coxeter group is defined as:

    $\sigma_i(r_1,\dots,r_m)=(r_1,\dots,r_{i-1},r_ir_{i+1}r_i^{-1},r_i,r_{i+2},\dots,r_m)$
\end{defi}

\begin{defi}
    Given a sequence $(r_1,\dots,r_n)$, the reverse map R is defined as $R(r_1,\dots,r_n)=(r_n,\dots,r_1)$. Given an exceptional sequence $(E_1,\cdots,E_n)$ in $mod\text-\Lambda$, the reverse map gives $R(E_1,\cdots,E_n)=(DE_n,\cdots,DE_1)$ an exceptional sequence in $mod\text-\Lambda^{op}$ where $DM=Hom_K(M,K)$.
\end{defi}

Since $DM$ has the same dimension vector as $M$, it corresponds to the same reflection.
If $(r_1,\dots,r_n)$ is a factorization of $c$, then $(r_n,\dots,r_1)$ is a factorization of $c^{-1}$. Thus $R$ is not an element of the braid group $B_n$. However, conjugation by $R$ is an automorphism of $B_n$ sending $\sigma_i$ to $\sigma_{n-i}^{-1}$. So, $R$ is an element of the semi-direct product $B_n\rtimes \mathbb Z_2$ which we denote by $\widetilde B_n$ and call the \textbf{augmented braid group}.

\begin{prop}
The augmented braid group $\widetilde B_n=B_n\rtimes \mathbb Z_2$ acts on the disjoint union of minimal factorizations of $c$ and $c^{-1}$.
\end{prop}

\begin{proof}
We check that the action of $R\sigma_iR$ on minimal factorizations is the same as the action of $\sigma_{n-i}^{-1}$:
\begin{equation}
\begin{split}
R\sigma_iR(s_1,\cdots,s_n)&= R\sigma_i(r_n,\cdots,r_1)\\
	&= R(r_n,\cdots,r_{n-i+1},r_{n-i}r_{n-i-1}r_{n-i},r_{n-i},r_{n-i},r_{n-i-2},\cdots,r_1)\\
	&=(r_1,\cdots,r_{n-i-2},r_{n-i},r_{n-i}r_{n-i-1}r_{n-i},r_{n-i+1},\cdots,r_n)\\
	&= \sigma_{n-i}^{-1}(r_1,\cdots,r_n)
\end{split}
\end{equation}
\end{proof}

\begin{defi}
    The \textbf{fundamental braid} is an element $\delta$ in the braid group defined as $\delta=\delta_n=\sigma_1\dots\sigma_{n-1}$. The \textbf{Garside element} of the braid group $B_n$ is given by $\Delta=\delta_{n}\delta_{n-1}\dots\delta_2$. The \textbf{reverse Garside element} $\Delta'$ of the augmented braid group $\widetilde B_n$ is given by $R \circ \Delta$. 
    For a minimal factorization $f=(r_1,\dots,r_n)$, we call $\Delta'(f)=f'$ the reverse dual of $f$.
\end{defi}

\begin{exa}
$\Delta'((12),(24),(23))=((12),(12)(24)(12),(12)(24)(23)(24)(12))=((12),(14),(34))$
\end{exa}

The reverse Garside element action was studied by Apostolakis and Ojakian in symmetric groups \cite{apostolakis2018duality}\cite{ojakian2019combinatorial}. Here we consider it in the setting of finite reflection groups instead of symmetric groups.

The following propositions were proved by Apostolakis and Ojakian in \cite{apostolakis2018duality} \cite{ojakian2019combinatorial} in the case of symmetric groups. The proof for general Coxeter groups is similar.

\begin{prop}
    Given a minimal factorization $(r_1,\dots,r_n)$ of a Coxeter element $c$. $\Delta'(r_1,\dots,r_n)=(r_1',\dots,r_n')$ is a factorization for $c^{-1}$, i.e, $r_1'\dots r_n'=c^{-1}$
\end{prop}

\begin{proof}
    \begin{equation}
    \begin{split}
    r_1'\dots r_n' & =r_1(r_1 r_2 r_1^{-1})\dots (r_1r_2\dots r_n \dots r_2^{-1} r_1^{-1}) \\
    & =r_2(r_2 r_3 r_2^{-1})\dots (r_2r_3\dots r_n \dots r_3^{-1} r_2^{-1})r_1 \\
    & =r_n r_{n-1} \dots r_1=c^{-1} \\
    \end{split}
    \end{equation}
\end{proof}

\begin{prop}[\cite{apostolakis2018duality}\cite{ojakian2019combinatorial}]
    $(\Delta')^2=id$
\end{prop}

\begin{proof}
    Consider a factorization $(r_1,\dots,r_n)$ of $c$. 
    
    Denote $(\Delta')^2(r_1,\dots,r_n)=(r_1^\#,\dots,r_n^\#)$

    \begin{equation}
    \begin{split}
    r_i^\# & = r_1' r_2' \dots r_i' \dots r_2' r_1' \\
    & =r_{i-1} \dots r_{1} r_i' r_{1} \dots r_{i-1} \\
    & =r_{i-1} \dots r_{1} (r_1 \dots r_{i-1} r_i r_{i-1} \dots r_1) r_1 \dots r_{i-1} \\
    & =r_i
    \end{split}
    \end{equation}    
\end{proof}

\subsection{Reverse Garside element action preserves relative projectivity}

In this section, we will prove that reverse Garside element action preserves relative projectivity.

\begin{thm}
    Given an exceptional sequence $(E_1,\dots,E_n)$, its Garside $\Delta(E_1,\dots,E_n)=(E_n',\dots,E_1')$ and its reverse Garside dual $\Delta'(E_1,\dots,E_n)=(DE_1',\dots,DE_n')$. $E_i$ is relatively projective if and only if $DE_i'$ is relatively projective.

    In particular, the reverse Garside element action $\Delta'$ gives a bijection between signed exceptional sequences of opposite categories that preserves relative projectivity. 
\end{thm}

\begin{proof}
    Suppose $E_i$ is relatively projective, then $E_i$ is projective in $\{E_{i+1},\dots,E_n\}^\perp$. Let 
    $\Delta(E_1,\dots,E_n)=(E_n',\dots,E_1')$. Then $E_i'$ is injective in $\{E_n',\dots,E_{i+1}'\}^\perp$. So, $DE_i'$ is projective in $\,^\perp\{DE_{i+1}',\dots,DE_n'\}$. Then $DE_i'$ is relatively projective in $\Delta'(E_1,\dots,E_n)=(DE_1',\dots,DE_n')$.
    
    The other direction can be obtained by applying $\Delta'$ again since $(\Delta')^2=id$
\end{proof}

We can obtain a similar result for minimal factorizations and one-way reflections.

\begin{cor}\label{thm B}
    Given a minimal factorization $(r_1,\dots,r_n)$ and its reverse Garside dual $(r_1',\dots,r_n')$. $r_i$ is one-way if and only if $r_i'$ is one-way.

    In particular, the reverse Garside element action $\Delta'$ gives a bijection between signed minimal factorizations of inverse Coxeter elements that preserves one-wayness. 
\end{cor}

\subsection{Equivalence of reverse Garside dual and planar dual}

\begin{defi}
    Given a cycle $(e_1\dots e_k)$ in the cycle type of a Coxeter element $c$. Denote \textbf{the arc associated to $e_i$}, $arc(e_i)$, the arc between $e_i$ and $e_{i+1}$ on the circle corresponding to the cycle $(e_1 \dots e_k)$ in the chord diagram. Denote \textbf{the region associated to $e_i$}, $reg(e_i)$, the unique region in the chord diagram that contains $arc(e_i)$ on the boundary.
\end{defi}

In general, in the chord diagrams, elements of the long cycle are put on the outer circle clockwise. So $arc(e_i)$ is the arc at clockwise position of $e_i$ on the circle.

For type $D$ specifically, the elements of the short cycle are put on the inner circle counterclockwise. So the $arc(e_i)$ in type $D$ on the inner circle is the arc at counterclockwise position of $e_i$ for $e_i$ in the short cycle.

Therefore, there's a bijection between vertices in the chord diagram, $e_i$, and regions in the chord diagram, $reg(e_i)$.

\begin{defi}
    Given a minimal factorization $(r_1,\dots,r_n)$ and its reverse Garside dual $(r_1',\dots,r_n')$ where $r_i'=r_1 \dots r_{i-1} r_i r_{i-1}^{-1} \dots r_1^{-1}$. Denote \textbf{trail of $r_i$} to be the set $Trail(r_i)$ consisting of $r_i$ and \ki{$r_{i_j}$ for}{} $i_j < i$ such that
    
    $$r_{i_j}r_{i_j+1}\dots r_{i-1} r_i r_{i-1}^{-1} \dots r_{i_j+1}^{-1}r_{i_j}^{-1} \neq r_{i_j+1}\dots r_{i-1} r_i r_{i-1}^{-1} \dots r_{i_j+1}^{-1}$$. 
\end{defi}

We can see that chords corresponding to trail of $r_i$ forms a path in the chord diagram that contains $r_i$.

We can also see that chords corresponding to $r_i'$ and $Trail(r_i)$ shares the same end-points in the chord diagram. 

\begin{thm}\label{thm C}
    Given a Coxeter element c and a minimal factorization $(r_1,\dots,r_n)$. Two regions $reg(e_i)$ and $reg(e_j)$ are adjacent in the planar dual if and only if there's a $r_k'$ in the reverse Garside dual such that $r_k'(e_i)=e_j$
\end{thm}

\begin{proof}
    Prove by induction on $r_i$ in the factorization.
    
    It holds for $r_1$ since there's only one chord in the chord diagram.

    Assume it's true for $r_{i-1}$, prove it for $r_i$.

    Let $Trail(r_i)=\{r_{i_1},\dots,r_{i_j},r_i\}$ where $i_j < i$.

    By assumption, each $r_{i_j}$ gives the adjacency for the two regions that contains $r_{i_j}$ on the boundary.

    Denote the regions that contain $r_{i+1}$ on the boundary as $reg(e_i)$ and $reg(e_j)$

    $Trail(r_i)$ contains all the chords on the boundary of $reg(e_i)$ and $reg(e_j)$ in the chord diagram of $W_{r_1,\dots,r_i}$

    Therefore, $r_i$ gives adjacency of $reg(e_i)$ and $reg(e_j)$ in the planar dual.

    Since $Trail(r_i)$ is given by the reverse Garside dual, this proves the theorem.
\end{proof}

\subsection{Edge-colored planar dual}
\begin{prop}
    A chord in the chord diagram has the same color as its corresponding chord in the planar dual
\end{prop}

This means, given a edge-colored chord diagram, we can construct its edge-colored dual purely using information from the chord diagram.

\begin{exa}
    Figure 8 shows the edge-colored chord diagram and the edge-colored planar dual for the factorization $(e_2-e_1,e_4-e_2,e_3-e_2)$. The chord diagram consists of chords $C_{1,2}$, $C_{2,4}$ and $C_{2,3}$ of which $C_{1,2}$ and $C_{2,4}$ are colored blue since they are one-way and $C_{2,3}$ is colored black.

    Denote the chords in the planar dual as $C_{i,j}'$. The planar dual consists of chords $C_{1,2}'$, $C_{1,4}'$ and $C_{3,4}'$ of which $C_{1,2}'$ and $C_{1,4}'$ are colored yellow since they are one-way and $C_{3,4}$ is colored red since it's two-way.
\end{exa}

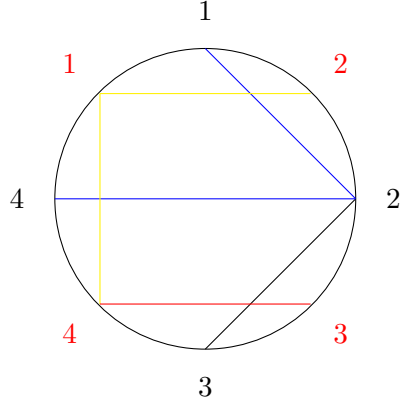
\begin{figure}
    \centering
    \begin{tikzpicture}
    \draw (0,0) circle (2cm);
  
    \node at (0,2.5) {1};
    \node at (2.5,0) {2};
    \node at (0,-2.5) {3};
    \node at (-2.5,0) {4};
    \node[text=red] at (1.8,1.8) {2};
    \node[text=red] at (1.8,-1.8) {3};
    \node[text=red] at (-1.8,-1.8) {4};
    \node[text=red] at (-1.8,1.8) {1};
  
    \draw[blue] (2,0) -- (0,2);
    \draw (2,0) -- (0,-2);
    \draw[blue] (2,0) -- (-2,0);
    \draw[yellow] (1.4, 1.4) -- (-1.4,1.4);
    \draw[yellow] (-1.4, 1.4) -- (-1.4, -1.4);
    \draw[red] (-1.4, -1.4) -- (1.4, -1.4);
    \end{tikzpicture}
    \caption{Given the factorization $(e_1-e_2,e_2-e_4,e_2-e_3)$. One-way reflections of the chord diagram in blue and one-way reflections of the planar dual in yellow}
\end{figure}

\section{Generalized Goulden-Yong duals}

Goulden and Yong constructed a bijection between the set of minimal factorizations in the symmetric group $S_n$ and the set of labelled trees on n vertices in \cite{MR1897927}. In this section, we will first review the Goulden-Yong duals and generalize their construction to type B and type D reflection groups.

\subsection{Goulden-Yong duals}

We will first review the Goulden-Yong duals.

\begin{defi}\cite{MR1897927}
    \textbf{Goulden-Yong dual} is constructed as follows:
    
    -Given a Coxeter element in $A_n$, construct its chord diagram

    -Place a vertex in each region of the chord diagram
    
    -Place an edge between two vertices if the boundaries of their corresponding regions share a chord

    -(Temporarily) label the edge by the chord it crosses

    -Label the vertex corresponding to the region $R(1)$ by n+1

    -For each edge, find the unique path to the vertex $n+1$ and slide the temporary label on the edge to the incident vertex away from $n+1$.
\end{defi}

\begin{figure}
    \centering
    \begin{tikzpicture}[baseline={(0,0)}]
    \draw (0,0) circle (2cm);
  
    \node at (0,2.5) {1};
    \node at (2.5,0) {2};
    \node at (0,-2.5) {3};
    \node at (-2.5,0) {4};
    \node[text=red] at (1.8,1.8) {2};
    \node[text=red] at (1.8,-1.8) {3};
    \node[text=red] at (-1.8,-1.8) {4};
    \node[text=red] at (-1.8,1.8) {1};
    \node at (0.8, 0.8) {1};
    \node at (0.8, 0.2) {2};
    \node at (0.8, -0.8) {3};
    \node[text=red] at (0,1) {1};
    \node[text=red] at (-1,0.2) {2};
    \node[text=red] at (0,-1) {3};
  
    \draw (2,0) -- (0,2);
    \draw (2,0) -- (0,-2);
    \draw (2,0) -- (-2,0);
    \draw[red] (1.4, 1.4) -- (-1.4,1.4);
    \draw[red] (-1.4, 1.4) -- (-1.4, -1.4);
    \draw[red] (-1.4, -1.4) -- (1.4, -1.4);
    \end{tikzpicture}
    \begin{tikzpicture}[baseline={(0,0)}]
    \node[text=red] at (1.8,1.8) {1};
    \node[text=red] at (1.8,-1.8) {3};
    \node[text=red] at (-1.8,-1.8) {2};
    \node[text=red] at (-1.8,1.8) {4};
  
    \draw[red] (1.4, 1.4) -- (-1.4,1.4);
    \draw[red] (-1.4, 1.4) -- (-1.4, -1.4);
    \draw[red] (-1.4, -1.4) -- (1.4, -1.4);
    \end{tikzpicture}
    \caption{Picture on the left is the chord diagram and planar dual of the factorization $(e_1-e_2,e_2-e_4,e_2-e_3)$. Picture on the right is the Goulden-Yong dual of this factorization}
\end{figure}
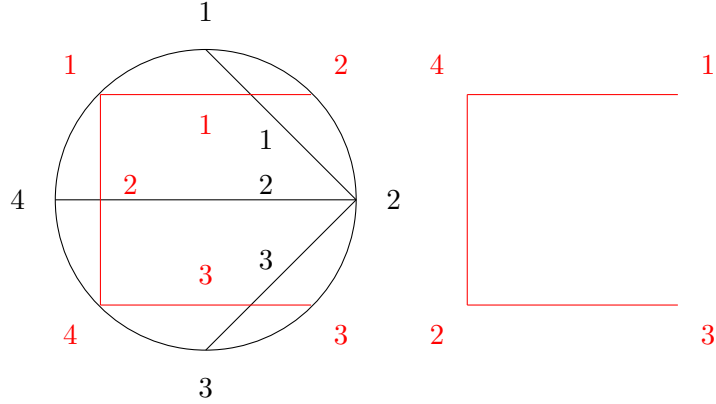

\begin{exa}
    Figure 9 is an example of the Goulden-Yong dual. Consider the factorization $(e_1-e_2,e_2-e_4,e_2-e_3)$. To get the Goulden-Yong dual, first change the vertex label 1 to 4. Then push the edge label to the vertex away from 4.
\end{exa}

\begin{rema}
    We can see that before assigning labels to vertices, the construction of Goulden-Yong duals agrees with the construction of planar duals. So we will construct type B,D Goulden-Yong duals from planar duals for type B,D.
\end{rema}

The important consequence of Goulden-Yong duals is that it gives a bijection between the set of minimal factorizations in $S_n$ and the set of trees on $n$ vertices.

\begin{thm}[Theorem 1.1 in \cite{MR1897927}]
    Let $F_n$ be the set of minimal factorizations for a Coxeter element in $A_n$ and $T_n$ be the set of labeled trees on $n+1$ vertices. Goulden-Yong dual gives a bijection between $f: F_n \rightarrow T_n$ such that a reflection $r_i$ in the factorization corresponds to a vertex $i$ in the tree where $i\le n$.
\end{thm}

\subsection{Folded chord diagram}
In order to define Goulden-Yong duals for type B and D, we will first define a variation of the chord diagram called folded chord diagrams. They share some properties with the chord diagrams of type A which allows us to build Goulden-Yong duals from them.

\subsubsection{Type B folded chord diagrams}

We will define the type B folded chord diagrams and their folded duals in this section.

\begin{defi}
Given a Coxeter element in $B_n$ of the form $(k_1 \dots k_{2n})$ where $k_i \in [n] \cup \{-[n]\}$ and $k_i \neq k_j$ and a minimal factorization $(r_1,\dots,r_n)$. 

\textbf{Folded circle chord diagram of type B} is a circle where $|k_i|$ are put on the circle clockwise for $1 \leq i \leq n$ 

Draw a chord between $|k_i|$ and $|k_j|$ if there's a reflection $r_k$ in the factorization such that $|r_k(k_i)|=|k_j|$. Label each chord by k if it corresponds to the reflection $r_k$. In particular, if $r_k$ corresponds to a loop, label the loop by k.
\end{defi}

\begin{rema}
    Each chord in the folded diagram between $i$ and $j$ can either be $e_i + e_j$ or $e_i-e_j$. However, if we fix a Coxeter element, there's only one choice for each chord in a minimal factorization. So we don't need to distinguish between the two choices for the chords in the folded chord diagram.
\end{rema}

Chords in type A chord diagrams form a tree for each factorization. We have similar properties in type B folded chord diagrams.

\begin{prop}
    The chords in a type B folded chord diagram form a tree with a loop attached to one of the vertices.
\end{prop}

\begin{proof}
    Each factorization in $B_n$ contains a root of form $e_i$ and the rest of the roots form a factorization of $A_{n-1}$. The roots that form a factorization of $A_{n-1}$ will form a tree in the folded chord diagram and the root of form $e_i$ will form a loop.
\end{proof}

The other important property of type A chord diagrams is that the edges labels for the chords around each vertex are clockwise decreasing. Type B folded chord diagrams also have this property which isn't held by type B chord diagrams.

\begin{prop}
    The chords in a type B folded chord diagram can be drawn in a way such that the edge labels on the chord encountered when moving around a vertex clockwise across the interior of the circle, form a decreasing sequence of elements in $\{1,\dots,n\}$
\end{prop}

\begin{proof}
    For type $B$ minimal factorizations, there's a single $r_k$ in the factorization such that $r_k(k_i)=-k_i$ which we draw it as a loop on $k_i$

    For reflections other than $r_k$ in the factorization, it forms a factorization of $A_{n-1}$ so they satisfy this property.
    
    For the reflection $r_k$ that corresponds to the loop at $k_i$, if there are reflections $r_i$ and $r_j$ in the factorization where $i < k < j$ such that $|r_i(k_i)| \neq k_i$, $|r_j(k_i)| \neq k_i$ and there's no reflection $r_l$ in the factorization such that $i < l < j$ and $|r_l(k_i)| \neq k_i$. Then draw the loop in the unique region bounded by the chords corresponds to $r_i$ and $r_j$

    The other cases for the reflection $r_k$ that corresponds to the loop at $k_i$ is when there's a reflection $r_i$ in the factorization where $i < k (\text{or} \ k < i)$ such that $|r_i(k_i)| \neq k_i$ and there's no reflection $r_l$ in the factorization such that $i < l < k (k < l < i)$ and $|r_i(k_i)| \neq k_i$. There are two regions that contains $r_i$ on the boundary. Pick the region such that the edge labels around the vertex $k_i$ are decreasing clockwise.
\end{proof}

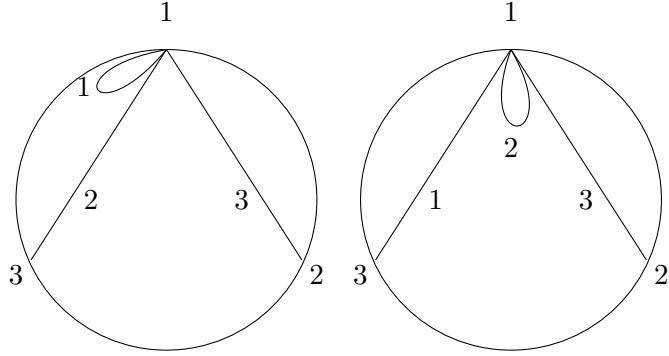
\begin{figure}
    \centering
    \begin{tikzpicture}[baseline={(0,0)}]
    \draw (0,0) circle (2cm);
  
    \node at (0,2.5) {1};
    \node at (2,-1) {2};
    \node at (-2,-1) {3};
    \node at (-1.1, 1.5) {1};
    \node at (1,0) {3};
    \node at (-1,0) {2};
    
    \draw (0,2) -- (1.8,-0.8);
    \draw (0,2) -- (-1.8,-0.8);
    \draw[scale=3] (0,0.666) to [out=190,in=230,loop] (0,0.666);
    \end{tikzpicture}
    \begin{tikzpicture}[baseline={(0,0)}]
    \draw (0,0) circle (2cm);
  
    \node at (0,2.5) {1};
    \node at (2,-1) {2};
    \node at (-2,-1) {3};
    \node at (1,0) {3};
    \node at (-1,0) {1};
    \node at (0, 0.7) {2};
    
    \draw (0,2) -- (1.8,-0.8);
    \draw (0,2) -- (-1.8,-0.8);
    \draw[scale=3] (0,0.666) to [out=250,in=300,loop] (0,0.666);
    \end{tikzpicture}
    \caption{The picture on the left is type B folded chord diagram for the factorization $(e_1,e_1-e_3,e_1-e_2)$ and the picture on the right is type B folded chord diagram for the factorization $(e_1+e_3,e_1,e_1-e_2)$}
\end{figure}

\begin{exa}
    Figure 10 shows two examples of type B folded chord diagrams. Both are factorizations of the Coxeter element of cycle type $(1,2,3,-1,-2,-3)$. 

    The picture on the left is for the factorization $(e_1,e_1-e_3,e_1-e_2)$. Since the loop has edge label 1 and the chord $C_{1,3}$ is the closest edge incident to 1. Hence, we draw the loop in the region bounded by $C_{1,3}$.

    The picture on the right is for the factorization $(e_1+e_3,e_1,e_1-e_2)$. Since the loop has edge label 2 and the chords $C_{1,2}$ and $C_{1,3}$ are the closest edges incident to 1. Hence, we draw the loop in the region bounded by $C_{1,3}$ and $C_{1,2}$
\end{exa}

Now we can define folded dual on the type B folded chord diagrams by drawing the reverse Garside duals on the folded chord diagrams.

\begin{defi}
    Given a factorization $(r_1,\dots,r_n)$ of $c$ with cycle type $(k_1 \dots k_{2n})$ where $k_i \in [n] \cup \{-[n]\}$ and $k_i \neq k_j$ and its reverse Garside dual $(r_1',\dots,r_n')$. \textbf{Type B Folded dual} is defined as follows:

    -Draw the folded chord diagram of $(r_1,\dots,r_n)$

    -Add a vertex $k_i$ on the arc between $k_i$ and ($k_{i+1}$ mod n) for $1 \leq 1 \leq n$

    -Draw a chord between $k_i$ and $k_j$ if there's a reflection $r_k'$ in the reverse Garside dual such that $|r_k'(k_i)|=|k_j|$.
\end{defi}

The following proposition shows that the folded dual constructed from reverse Garside dual can be constructed graphically just like constructing planar duals from chord diagrams.

\begin{prop}
    Type B folded dual can be constructed from a type B folded chord diagram as follows:
    
    (i) Ignore the loop from the folded chord diagram and draw the planar dual of the tree.

    (ii) There's a unique vertex of the folded dual in the region of the loop of the folded chord diagram, say $v_i$. Draw loop of the folded dual by drawing a loop from $v_i$ into the region bounded by the loop in the folded chord diagram and back to $v_i$
\end{prop}

\begin{proof}
    Since the tree part of the folded chord diagram corresponds to a factorization of $A_{n-1}$, (i) follows.

    For (ii), suppose we have a minimal factorization $(r_1, \dots,r_i,\dots,r_k, \dots, r_j, \dots, r_n)$ of $c$ where $r_k$ corresponds to the loop in the folded chord diagram and the loop lies in the region bounded by chords corresponding to $r_i$ and $r_j$.

    Consider its dual $(r_1',\dots,r_i',\dots,r_k', \dots, r_j', \dots, r_n')$. $r_k'$ corresponds to the loop in the folded dual and lies in the region bounded by $r_i'$ and $r_j'$.

    Since the dual is a factorization of $c^{-1}$, the edge labels around a vertex is clockwise increasing. Therefore, loops corresponding to $r_k$ and $r_k'$ lies in the region bounded by $r_i, r_j, r_i', r_j'$.

    A similar argument can be made when $r_k$ is only bounded one-side.

    Therefore, the loop of the folded dual lies in the region of the folded chord diagram.

\end{proof}

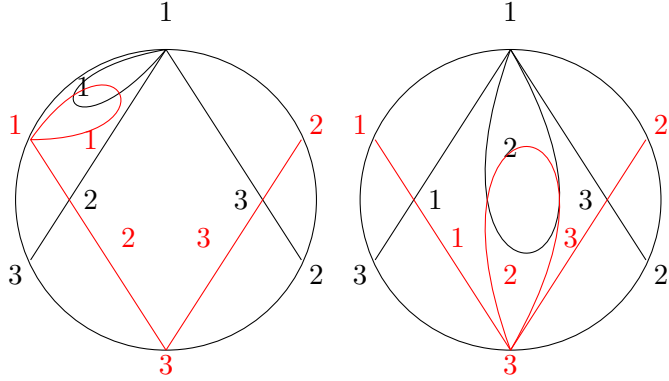
\begin{figure}
    \centering
    \begin{tikzpicture}[baseline={(0,0)}]
    \draw (0,0) circle (2cm);
  
    \node at (0,2.5) {1};
    \node at (2,-1) {2};
    \node at (-2,-1) {3};
    \node at (-1.1, 1.5) {1};
    \node at (1,0) {3};
    \node at (-1,0) {2};
    \node[red] at (2,1) {2};
    \node[red] at (0,-2.2) {3};
    \node[red] at (-2, 1) {1};
    \node[red] at (-1,0.8) {1};
    \node[red] at (-0.5,-0.5) {2};
    \node[red] at (0.5, -0.5) {3};
    
    \draw (0,2) -- (1.8,-0.8);
    \draw (0,2) -- (-1.8,-0.8);
    \draw[scale=4] (0,0.5) to [out=190,in=230,loop] (0,0.5);
    \draw[red] (1.8,0.8) -- (0,-2);
    \draw[red] (-1.8, 0.8) -- (0, -2);
    \draw[scale=4, red] (-0.45,0.2) to [out=0,in=55,loop] (-0.45,0.2);
    \end{tikzpicture}
    \begin{tikzpicture}[baseline={(0,0)}]
    \draw (0,0) circle (2cm);
  
    \node at (0,2.5) {1};
    \node at (2,-1) {2};
    \node at (-2,-1) {3};
    \node at (1,0) {3};
    \node at (-1,0) {1};
    \node at (0, 0.7) {2};
    \node[red] at (2,1) {2};
    \node[red] at (0,-2.2) {3};
    \node[red] at (-2, 1) {1};
    \node[red] at (0,-1) {2};
    \node[red] at (-0.7,-0.5) {1};
    \node[red] at (0.8, -0.5) {3};
    
    \draw (0,2) -- (1.8,-0.8);
    \draw (0,2) -- (-1.8,-0.8);
    \draw[scale=8] (0,0.25) to [out=250,in=300,loop] (0,0.25);
    \draw[red] (1.8,0.8) -- (0,-2);
    \draw[red] (-1.8, 0.8) -- (0, -2);
    \draw[scale=8, red] (0,-0.25) to [out=-250,in=-300,loop] (0,-0.25);
    \end{tikzpicture}
    \caption{The picture on the left is type B folded chord diagram for the factorization $(e_1,e_1-e_3,e_1-e_2)$ and the picture on the right is type B folded chord diagram for the factorization $(e_1+e_3,e_1,e_1-e_2)$}
\end{figure}

\subsubsection{Type D folded chord diagram}
We will define the type D folded chord diagrams and their folded duals in this section.

\begin{defi}
Given a Coxeter element of $D_n$ of cycle type $(k_1 \dots k_{2n-2})(1,-1)$ where $k_i \in \{\pm2, \pm3, \dots, \pm n\}$ and a minimal factorization $(r_1,\dots,r_n)$. 

\textbf{Type D folded chord diagrams} contain two circles where one circle contains the other. $|k_i|$ are put on the outside circle clockwise for $1 \leq i \leq {n-1}$ and $\{1\}$ is put on the inside counterclockwise.

For $a,b \in [n]$. Draw a chord between $a$ and $b$ if there's a reflection $r_k$ in the factorization such that $|r_k(k_i)|=|k_j|$ in the unique homotopy class such that the resulting graph remains planar and edge labels around vertex 1 are clockwise decreasing. Label each chord by k if it corresponds to the reflection $r_k$.
\end{defi}

It is known from previous constructions that edges labels around vertices on the outside circle can be drawn to be clockwise decreasing. For each chord connected to vertex 1 on the inside circle, there's exactly one homotopy class out of the possible 2 homotopy classes such that the chords have clockwise decreasing edge labels and don't intersect each other.

\begin{rema}
    Here we take $1$ to be in the short cycle. In general, any number can be in the short cycle.
\end{rema}

For type D, when we draw factorizations in the folded chord diagrams, we actually lost information on factorizations which results in one folded chord diagram corresponding to two factorizations.

\begin{lemma}
    Each type D folded chord diagram corresponds to two minimal factorizations of a given Coxeter element.
\end{lemma}

\begin{proof}
    Each chord corresponds to reflections either in the form of $e_i+e_j$ or $e_i-e_j$. If a reflection is in the form $e_i+e_j$, we say it has $-$ sign. If a reflection is in the form $e_i-e_j$, we say it has $+$ sign

    For chords not incident to 1, their signs are uniquely determined since the vertices other than 1 are in the same cycle in the cycle type of type D Coxeter elements.

    For chords incident to 1, their signs are determined by a choice of signs for one of its incident chords. If we choose a different sign of one of the chords incident to 1, signs of all the other chords incident to 1 need to change as well.

    Here, we can choose the chord that has the biggest edge label among all edges incident to 1.

    Therefore, the signs of the reflections in a factorization is determined by the sign of the chord that has the biggest edge label among all edges incident to 1 and each type D folded chord diagram corresponds to two minimal factorizations where the chosen edge has either $+$ or $-$ sign.
\end{proof}

Similar as before, the chords in type D folded chord diagrams form certain graphs.

\begin{prop}
    The chords in a type D folded chord diagram form a spanning graph on n vertices with a single cycle of size at least 2 that contains the vertex 1.
\end{prop}

\begin{proof}
    Since the chords form a graph of $n$ edges on $n$ vertices, the graph must contain a cycle.

    The cycle is of size at least $2$ since there's no reflection of the form $e_i$ in $D_n$.

    The rest of the chords correspond to a forest attached to the cycle.
\end{proof}

Similar as before, type D folded chord diagrams also have the property that edges labels around a vertex are clockwise decreasing whereas type D chord diagrams don't have this property.

\begin{prop}
    The chords of a folded type D chord diagram can be drawn without crossings such that the edge labels on the chord encountered when moving around a vertex clockwise across the interior of the circle, form a decreasing sequence of elements in $\{1,\dots,n\}$
\end{prop}

\begin{proof}
    The chords outside the cycle can be drawn such that edge labels around a vertex are clockwise decreasing since they correspond to minimal factorization of type $A$.

    For a vertex labeled $k_i \neq 1$ that is contained in the cycle other than the vertex 1, there are exactly two chords in the cycle adjacent to $v_i$, say $C_{r_i}$ and $C_{r_j}$ where $i < j$. 

    There's no $r_k$ where $i < k < j$ and $r_k(i) \neq (i)$. Otherwise, the corresponding Coxeter element doesn't contain the short cycle $(1,-1)$

    Therefore, we can find $r_l$ and $r_m$ in the factorization such that $l < i < j < m$ and $r_l(i) \neq i$, $r_m(i) \neq i$.

    If there are reflections $r_l$ and $r_m$ in the factorization where $l < i < j < m$ such that $r_l(k_i) \neq k_i$, $r_m(k_i) \neq k_i$ and there's no reflection $r_p$ in the factorization such that $l < p < m$ and $r_l(k_i) \neq k_i$. Then draw the $r_i$ and $r_j$ clockwise decreasingly in the unique region bounded by  $C_{r_l}$ and $C_{r_m}$

    The other cases are when there's a reflection $r_i$ in the factorization where $l < i < j (\text{or} \ i < j < l)$ such that $r_l(k_i) \neq k_i$ and there's no reflection $r_m$ in the factorization such that $l < m (m < l )$ and $r_l(k_i) \neq k_i$. There are two regions that contain $r_i$ on the boundary. Pick the region such that the edge labels around the vertex $k_i$ are decreasing clockwise.

    For the two chords inside the cycle that are incident to 1, draw each chord in the region such that the edge labels are clockwise decreasing.

    Fix a chord $E$ joining the inner and outer boundaries.
    For a chord $F$ joining the inner vertex to a different
    outer vertex, there are two isotopy classes relative
    to the endpoints with representatives whose interiors
    are disjoint from $E$. Indeed, cutting along $E$ gives
    a disk with two boundary copies of the inner vertex.
    The two classes correspond to joining the outer endpoint
    of $F$ to these two copies. The required clockwise order
    of $E$ and $F$ at the inner vertex selects one class.

    This drawing is unique up to homeomorphisms (Dehn twists) fixing the boundary circles since homeomorphisms won't change the edge label being clockwise decreasing.
\end{proof}

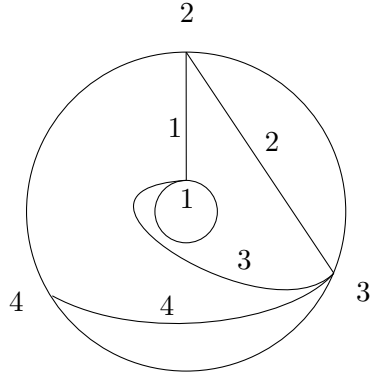
\begin{figure}
    \centering
    \begin{tikzpicture}[x=0.75pt,y=0.75pt,yscale=-1,xscale=1]

\draw   (251,154.5) .. controls (251,110.04) and (287.04,74) .. (331.5,74) .. controls (375.96,74) and (412,110.04) .. (412,154.5) .. controls (412,198.96) and (375.96,235) .. (331.5,235) .. controls (287.04,235) and (251,198.96) .. (251,154.5) -- cycle ;
\draw   (315.75,154.43) .. controls (315.79,145.73) and (322.87,138.71) .. (331.57,138.75) .. controls (340.27,138.79) and (347.29,145.87) .. (347.25,154.57) .. controls (347.21,163.27) and (340.13,170.29) .. (331.43,170.25) .. controls (322.73,170.21) and (315.71,163.13) .. (315.75,154.43) -- cycle ;
\draw    (331.5,74) -- (406,185.6) ;
\draw    (331.57,138.75) .. controls (252,143.6) and (373,218.6) .. (406,185.6) ;
\draw    (331.5,74) -- (331.57,138.75) ;
\draw    (264,197) .. controls (309,221.6) and (384,211.2) .. (406,185.6) ;

\draw (327,142) node [anchor=north west][inner sep=0.75pt]   [align=left] {1};
\draw (327,48) node [anchor=north west][inner sep=0.75pt]   [align=left] {2};
\draw (416,189) node [anchor=north west][inner sep=0.75pt]   [align=left] {3};
\draw (241,194) node [anchor=north west][inner sep=0.75pt]   [align=left] {4};
\draw (370,113) node [anchor=north west][inner sep=0.75pt]   [align=left] {2};
\draw (356,173) node [anchor=north west][inner sep=0.75pt]   [align=left] {3};
\draw (321,106) node [anchor=north west][inner sep=0.75pt]   [align=left] {1};
\draw (317,196) node [anchor=north west][inner sep=0.75pt]   [align=left] {4};
\end{tikzpicture}
\caption{Folded chord diagram for factorizations $(e_1+e_2,e_2-e_3,e_1-e_3,e_3-e_4)$ and $(e_1-e_2,e_2-e_3,e_1+e_3,e_3-e_4)$ of Coxeter element of cycle type (2,3,4,-2,-3,-4)(1,-1)}
\end{figure}

\begin{exa}
    Figure 12 is the folded chord diagram for minimal factorizations $(e_1+e_2,e_2-e_3,e_1-e_3,e_3-e_4)$ and $(e_1-e_2,e_2-e_3,e_1+e_3,e_3-e_4)$ of Coxeter element of cycle type (2,3,4,-2,-3,-4)(1,-1). The sign of the chord labeled 3 determines which factorization it represents. If it has $+$ sign, then the folded chord diagram represents the factorization $(e_1+e_2,e_2-e_3,e_1-e_3,e_3-e_4)$. If it has $-$ sign, then the folded chord diagram represents the factorization $(e_1-e_2,e_2-e_3,e_1+e_3,e_3-e_4)$

    Fixing the sign, this drawing of the folded chord diagram such that each vertex has edge labels clockwise decreasing is unique up to homeomorphism fixing the boundary circles.
\end{exa}

Now we can define folded duals on type D folded chord diagrams.

\begin{defi}
    Given a factorization $(r_1,\dots,r_n)$ of $c$ with cycle type $(k_1, \dots k_{2n-2})(1, -1)$ where $k_i \in \{\pm2, \pm3, \dots, \pm n\}$ and $k_i \neq k_j$ and its reverse Garside dual $(r_1',\dots,r_n')$. \textbf{Folded dual} is defined as follows:

    -Draw the folded chord diagram of $(r_1,\dots,r_n)$

    -Put a vertex $k_i$ on the arc between $k_i$ and ($k_{i+1}$ mod n) for $1 \leq 1 \leq n-1$ on the outside circle

    -Put a vertex $1$ on the inside circle

    -Draw a chord between $k_i$ and $k_j$ if there's a reflection $r_k'$ in the reverse Garside dual such that $|r_k'(k_i)|=|k_j|$ for $k_i,k_j \in [n]$
\end{defi}

The following proposition shows that the folded dual constructed from reverse Garside dual can be constructed graphically just like constructing planar dual from chord diagrams.

\begin{prop}
    Folded dual can be constructed from a type D chord diagram as follows:
    
    -Draw the planar dual such that each chord in the folded dual only intersects one chord in the folded chord diagram
\end{prop}

\begin{proof}
    Since the chords can be drawn in a clockwise decreasing way, the folded dual is similar to the planar dual of type A chord diagrams.
\end{proof}

\begin{figure}
    \centering
\begin{tikzpicture}[x=0.75pt,y=0.75pt,yscale=-1,xscale=1]

\draw   (251,154.5) .. controls (251,110.04) and (287.04,74) .. (331.5,74) .. controls (375.96,74) and (412,110.04) .. (412,154.5) .. controls (412,198.96) and (375.96,235) .. (331.5,235) .. controls (287.04,235) and (251,198.96) .. (251,154.5) -- cycle ;
\draw   (315.75,154.43) .. controls (315.79,145.73) and (322.87,138.71) .. (331.57,138.75) .. controls (340.27,138.79) and (347.29,145.87) .. (347.25,154.57) .. controls (347.21,163.27) and (340.13,170.29) .. (331.43,170.25) .. controls (322.73,170.21) and (315.71,163.13) .. (315.75,154.43) -- cycle ;
\draw    (331.5,74) -- (406,185.6) ;
\draw    (331.57,138.75) .. controls (252,143.6) and (373,218.6) .. (406,185.6) ;
\draw    (331.5,74) -- (331.57,138.75) ;
\draw    (264,197) .. controls (309,221.6) and (384,211.2) .. (406,185.6) ;
\draw [color={rgb, 255:red, 208; green, 2; blue, 27 }  ,draw opacity=1 ]   (262,116) -- (331.5,235) ;
\draw [color={rgb, 255:red, 208; green, 2; blue, 27 }  ,draw opacity=1 ]   (337,169.6) -- (406,120.6) ;
\draw [color={rgb, 255:red, 208; green, 2; blue, 27 }  ,draw opacity=1 ]   (262,116) .. controls (302,86) and (291.43,200.25) .. (331.43,170.25) ;
\draw [color={rgb, 255:red, 208; green, 2; blue, 27 }  ,draw opacity=1 ]   (262,116) .. controls (302,86) and (387,136.6) .. (337,169.6) ;

\draw (325,140) node [anchor=north west][inner sep=0.75pt]   [align=left] {1};
\draw (327,48) node [anchor=north west][inner sep=0.75pt]   [align=left] {2};
\draw (416,189) node [anchor=north west][inner sep=0.75pt]   [align=left] {3};
\draw (241,194) node [anchor=north west][inner sep=0.75pt]   [align=left] {4};
\draw (370,113) node [anchor=north west][inner sep=0.75pt]   [align=left] {2};
\draw (356,173) node [anchor=north west][inner sep=0.75pt]   [align=left] {3};
\draw (321,106) node [anchor=north west][inner sep=0.75pt]   [align=left] {1};
\draw (317,196) node [anchor=north west][inner sep=0.75pt]   [align=left] {4};
\draw (413,107) node [anchor=north west][inner sep=0.75pt]  [color={rgb, 255:red, 208; green, 2; blue, 27 }  ,opacity=1 ] [align=left] {3};
\draw (242,102) node [anchor=north west][inner sep=0.75pt]  [color={rgb, 255:red, 208; green, 2; blue, 27 }  ,opacity=1 ] [align=left] {2};
\draw (324,241) node [anchor=north west][inner sep=0.75pt]  [color={rgb, 255:red, 208; green, 2; blue, 27 }  ,opacity=1 ] [align=left] {4};
\draw (325,155) node [anchor=north west][inner sep=0.75pt]  [color={rgb, 255:red, 208; green, 2; blue, 27 }  ,opacity=1 ] [align=left] {1};
\draw (296,91) node [anchor=north west][inner sep=0.75pt]  [color={rgb, 255:red, 208; green, 2; blue, 27 }  ,opacity=1 ] [align=left] {1};
\draw (283,134) node [anchor=north west][inner sep=0.75pt]  [color={rgb, 255:red, 208; green, 2; blue, 27 }  ,opacity=1 ] [align=left] {3};
\draw (271,153) node [anchor=north west][inner sep=0.75pt]  [color={rgb, 255:red, 208; green, 2; blue, 27 }  ,opacity=1 ] [align=left] {4};
\draw (366,147) node [anchor=north west][inner sep=0.75pt]  [color={rgb, 255:red, 208; green, 2; blue, 27 }  ,opacity=1 ] [align=left] {2};
\end{tikzpicture}
\caption{Folded dual and folded chord diagram for factorizations $(e_1+e_2,e_2-e_3,e_1-e_3,e_3-e_4)$ and $(e_1-e_2,e_2-e_3,e_1+e_3,e_3-e_4)$ of Coxeter element of cycle type (2,3,4,-2,-3,-4)(1,-1)}
\end{figure}
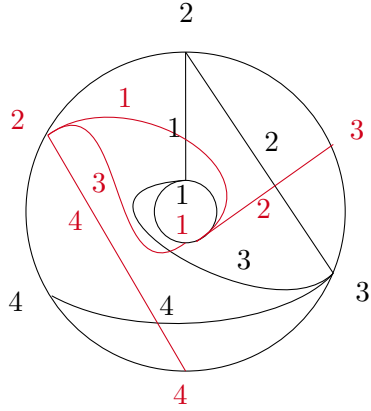

\begin{exa}
    Figure 13 shows how to draw folded dual from the folded chord diagram by drawing chords that will only intersect with one of chords in the folded chord diagram.
\end{exa}

\subsection{Generalized Goulden-Yong duals}
In this section, we will construct generalized Goulden-Yong duals for type B and type D reflection groups.

\subsubsection{Type B Goulden-Yong dual}
In this section, we will define the type B Goulden-Yong dual and proves that it gives a bijection between minimal factorizations in $B_n$ and rooted trees with a loop attached.

\begin{defi}
    \textbf{Type B Goulden-Yong dual} is constructed as follows:
    
    -Given a Coxeter element in $B_n$, construct its folded dual

    -Make the vertex labeled 1 as root

    -For the unique vertex $v_i$ that's adjacent to a loop, change the vertex label to the edge label of the loop

    -Remove the loop and the rest of the folded dual is a tree.

    -For each edge, find the unique path to the $v_i$ and slide the label on the edge to the incident vertex away from $v_i$ (This makes the new vertex labels equal to the old edge labels.)

    -Add back the loop
\end{defi}

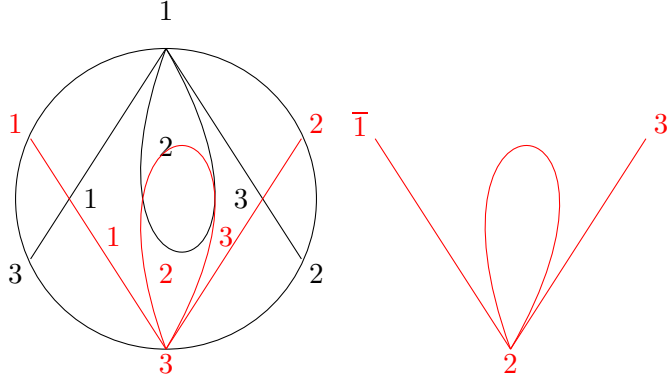
\begin{figure}
    \centering
    \begin{tikzpicture}[baseline={(0,0)}]
    \draw (0,0) circle (2cm);
  
    \node at (0,2.5) {1};
    \node at (2,-1) {2};
    \node at (-2,-1) {3};
    \node at (1,0) {3};
    \node at (-1,0) {1};
    \node at (0, 0.7) {2};
    \node[red] at (2,1) {2};
    \node[red] at (0,-2.2) {3};
    \node[red] at (-2, 1) {1};
    \node[red] at (0,-1) {2};
    \node[red] at (-0.7,-0.5) {1};
    \node[red] at (0.8, -0.5) {3};
    
    \draw (0,2) -- (1.8,-0.8);
    \draw (0,2) -- (-1.8,-0.8);
    \draw[scale=8] (0,0.25) to [out=250,in=300,loop] (0,0.25);
    \draw[red] (1.8,0.8) -- (0,-2);
    \draw[red] (-1.8, 0.8) -- (0, -2);
    \draw[scale=8, red] (0,-0.25) to [out=-250,in=-300,loop] (0,-0.25);
    \end{tikzpicture}
    \begin{tikzpicture}[baseline={(0,0)}]  
    \node[red] at (2,1) {3};
    \node[red] at (0,-2.2) {2};
    \node[red] at (-2, 1) {$\overline{1}$};
    
    \draw[red] (1.8,0.8) -- (0,-2);
    \draw[red] (-1.8, 0.8) -- (0, -2);
    \draw[scale=8, red] (0,-0.25) to [out=-250,in=-300,loop] (0,-0.25);
    \end{tikzpicture}
    \caption{Type B Goulden-Yong dual for $(e_1+e_3,e_1,e_1-e_2)$ where $1$ is the root.}
\end{figure}

\begin{exa}
    Figure 14 shows the type B Goulden-Yong dual for $(e_1+e_3,e_1,e_1-e_2)$.
    
    We use the folded dual to construct the Goulden-Yong dual.
    
    First mark the vertex labeled 1 as root.
    
    The unique vertex adjacent to a loop in the folded dual is $3$. Replace its label by the edge label of the loop which is 2.

    Ignore the loop and focus on the tree part. Push the edge labels away from $3$. So vertex 1 has label 1 and vertex 2 has label 3.

    So we have a rooted tree with a loop where the root is at $1$ and the loop is attached to $2$.
\end{exa}

Recall that Goulden-Yong duals give a bijection between the set of minimal factorizations in $A_n$ and the number of trees on $n+1$ vertices. Type B Goulden-Yong duals also gives a similar bijection. 

\begin{thm}\label{thm D}
    Type B Goulden-Yong duals give a bijection between set of minimal factorizations of a Coxeter element in $B_n$ and the set $BT_n$, the set of rooted trees on n vertices with a loop attached to one of the vertices, such that a reflection $r_i$ in the factorization corresponds to a vertex $i$ in the rooted tree.
\end{thm}

\begin{proof}
    Type B Goulden-Yong dual maps a factorization $(r_1,\dots,r_n)$ to a rooted tree with a loop $T$ where a vertex in the graph corresponds to an edge in the folded dual which corresponds to a reflection in the factorization.

    Given a rooted tree on n vertices with a loop. The inverse map is given as follows:

    -Find the unique vertex $v_i$ that's adjacent to a loop

    -Label the loop by the vertex label of $v_i$

    -For each vertex, find the unique path to $v_i$ and slide the label on the vertex to the incident edge on the path

    -Give the root label 1

    The inverse corresponds to a unique factorization because the tree part of the graph corresponds to a unique folded dual given by inverse of Goulden-Yong dual
\end{proof}

\subsubsection{Type D Goulden-Yong dual}
In this section, we will define the type D Goulden-Yong dual and proves that it gives a bijection between type D folded duals and rooted simple spanning graphs on $n$ vertices with $n$ edges where the root cannot be the smallest vertex in the cycle.
\begin{defi}
    \textbf{Type D Goulden-Yong dual} is constructed as follows:
    
    -Given a Coxeter element in $D_n$, construct its folded dual

    -Make the vertex labeled 2 as root

    -For each vertex inside the cycle except for 1, there are two edges in the cycle adjacent to it, choose the bigger edge label and use it to replace the vertex label

    -For the vertex labeled 1, choose the smaller edge label and use it to replace the vertex label

    -For each edge outside the cycle, slide the label on the edge to the incident vertex away from the cycle
\end{defi}

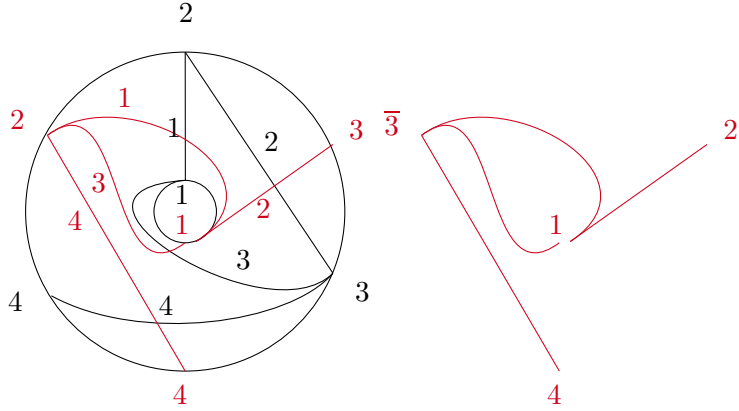
\begin{figure}
\centering
\begin{tikzpicture}[x=0.75pt,y=0.75pt,yscale=-1,xscale=1]

\draw   (251,154.5) .. controls (251,110.04) and (287.04,74) .. (331.5,74) .. controls (375.96,74) and (412,110.04) .. (412,154.5) .. controls (412,198.96) and (375.96,235) .. (331.5,235) .. controls (287.04,235) and (251,198.96) .. (251,154.5) -- cycle ;
\draw   (315.75,154.43) .. controls (315.79,145.73) and (322.87,138.71) .. (331.57,138.75) .. controls (340.27,138.79) and (347.29,145.87) .. (347.25,154.57) .. controls (347.21,163.27) and (340.13,170.29) .. (331.43,170.25) .. controls (322.73,170.21) and (315.71,163.13) .. (315.75,154.43) -- cycle ;
\draw    (331.5,74) -- (406,185.6) ;
\draw    (331.57,138.75) .. controls (252,143.6) and (373,218.6) .. (406,185.6) ;
\draw    (331.5,74) -- (331.57,138.75) ;
\draw    (264,197) .. controls (309,221.6) and (384,211.2) .. (406,185.6) ;
\draw [color={rgb, 255:red, 208; green, 2; blue, 27 }  ,draw opacity=1 ]   (262,116) -- (331.5,235) ;
\draw [color={rgb, 255:red, 208; green, 2; blue, 27 }  ,draw opacity=1 ]   (337,169.6) -- (406,120.6) ;
\draw [color={rgb, 255:red, 208; green, 2; blue, 27 }  ,draw opacity=1 ]   (262,116) .. controls (302,86) and (291.43,200.25) .. (331.43,170.25) ;
\draw [color={rgb, 255:red, 208; green, 2; blue, 27 }  ,draw opacity=1 ]   (262,116) .. controls (302,86) and (387,136.6) .. (337,169.6) ;

\draw (325,140) node [anchor=north west][inner sep=0.75pt]   [align=left] {1};
\draw (327,48) node [anchor=north west][inner sep=0.75pt]   [align=left] {2};
\draw (416,189) node [anchor=north west][inner sep=0.75pt]   [align=left] {3};
\draw (241,194) node [anchor=north west][inner sep=0.75pt]   [align=left] {4};
\draw (370,113) node [anchor=north west][inner sep=0.75pt]   [align=left] {2};
\draw (356,173) node [anchor=north west][inner sep=0.75pt]   [align=left] {3};
\draw (321,106) node [anchor=north west][inner sep=0.75pt]   [align=left] {1};
\draw (317,196) node [anchor=north west][inner sep=0.75pt]   [align=left] {4};
\draw (413,107) node [anchor=north west][inner sep=0.75pt]  [color={rgb, 255:red, 208; green, 2; blue, 27 }  ,opacity=1 ] [align=left] {3};
\draw (242,102) node [anchor=north west][inner sep=0.75pt]  [color={rgb, 255:red, 208; green, 2; blue, 27 }  ,opacity=1 ] [align=left] {2};
\draw (324,241) node [anchor=north west][inner sep=0.75pt]  [color={rgb, 255:red, 208; green, 2; blue, 27 }  ,opacity=1 ] [align=left] {4};
\draw (325,155) node [anchor=north west][inner sep=0.75pt]  [color={rgb, 255:red, 208; green, 2; blue, 27 }  ,opacity=1 ] [align=left] {1};
\draw (296,91) node [anchor=north west][inner sep=0.75pt]  [color={rgb, 255:red, 208; green, 2; blue, 27 }  ,opacity=1 ] [align=left] {1};
\draw (283,134) node [anchor=north west][inner sep=0.75pt]  [color={rgb, 255:red, 208; green, 2; blue, 27 }  ,opacity=1 ] [align=left] {3};
\draw (271,153) node [anchor=north west][inner sep=0.75pt]  [color={rgb, 255:red, 208; green, 2; blue, 27 }  ,opacity=1 ] [align=left] {4};
\draw (366,147) node [anchor=north west][inner sep=0.75pt]  [color={rgb, 255:red, 208; green, 2; blue, 27 }  ,opacity=1 ] [align=left] {2};
\end{tikzpicture}
\begin{tikzpicture}[x=0.75pt,y=0.75pt,yscale=-1,xscale=1]

\draw [color={rgb, 255:red, 208; green, 2; blue, 27 }  ,draw opacity=1 ]   (262,116) -- (331.5,235) ;
\draw [color={rgb, 255:red, 208; green, 2; blue, 27 }  ,draw opacity=1 ]   (337,169.6) -- (406,120.6) ;
\draw [color={rgb, 255:red, 208; green, 2; blue, 27 }  ,draw opacity=1 ]   (262,116) .. controls (302,86) and (291.43,200.25) .. (331.43,170.25) ;
\draw [color={rgb, 255:red, 208; green, 2; blue, 27 }  ,draw opacity=1 ]   (262,116) .. controls (302,86) and (387,136.6) .. (337,169.6) ;

\draw (413,107) node [anchor=north west][inner sep=0.75pt]  [color={rgb, 255:red, 208; green, 2; blue, 27 }  ,opacity=1 ] [align=left] {2};
\draw (242,102) node [anchor=north west][inner sep=0.75pt]  [color={rgb, 255:red, 208; green, 2; blue, 27 }  ,opacity=1 ] [align=left] {$\overline{3}$};
\draw (324,241) node [anchor=north west][inner sep=0.75pt]  [color={rgb, 255:red, 208; green, 2; blue, 27 }  ,opacity=1 ] [align=left] {4};
\draw (325,155) node [anchor=north west][inner sep=0.75pt]  [color={rgb, 255:red, 208; green, 2; blue, 27 }  ,opacity=1 ] [align=left] {1};

\end{tikzpicture}
\caption{Type D Goulden-Yong dual for factorizations $(e_1+e_2,e_2-e_3,e_1-e_3,e_3-e_4)$ and $(e_1-e_2,e_2-e_3,e_1+e_3,e_3-e_4)$ where $3$ is the root.}
\end{figure}

\begin{exa}
    Figure 15 shows how to construct type D Goulden-Yong dual for $(e_1+e_2,e_2-e_3,e_1-e_3,e_3-e_4)$ and $(e_1-e_2,e_2-e_3,e_1+e_3,e_3-e_4)$.

    First mark the vertex labeled 2 as root.

    Within the cycle, since $2$ is incident to edges $1$ and $3$, so replace the vertex label of $2$ by $3$. Since $1$ is incident to edges $1$ and $3$, replace the vertex label of $1$ by $1$.

    The root is now labeled $3$.

    Then push labels of edges outside of the cycle away from the cycle. So label vertex $3$ by $2$ and label vertex $4$ by 4.

    We've obtained the type D Goulden-Yong dual for $(e_1+e_2,e_2-e_3,e_1-e_3,e_3-e_4)$ and $(e_1-e_2,e_2-e_3,e_1+e_3,e_3-e_4)$.
\end{exa}

Like type B, type D Goulden-Yong duals gives a bijection between type D folded duals and rooted simple spanning graphs on $n$ vertices with n edges where the root cannot be the smallest vertex in the cycle. Since each type D folded dual corresponds to two minimal factorizations in $D_n$, it gives a 2-to-1 correspondence between minimal factorizations in $D_n$ and the set of graphs above.

\begin{thm}\label{thm E}
    Type D Goulden-Yong duals give a 2-to-1 correspondence between the set of minimal factorizations of a Coxeter element in $D_n$ and the set $DT_n$, the set of  rooted simple spanning graphs on $n$ vertices with n edges where the root cannot be the smallest vertex in the cycle, such that a reflection $r_i$ in the factorization corresponds to a vertex $i$ in the rooted tree.
\end{thm}

\begin{proof}
        Similar to type B, a vertex in the graph corresponds to an edge in the folded dual which corresponds to a reflection in the factorization.

    The inverse map from type D Goulden-Yong duals to type D folded duals is given as follows:

    -Within the cycle, choose the cyclic orientation such that the vertex label increases along the cyclic orientation starting from the smallest vertex label and slide the vertex label to the edge adjacent to it according to the cyclic orientation

    -For each vertex outside of the cycle, there's a unique path to the cycle. Slide the vertex label to the edge adjacent to it on the path

    -Change the vertex label of the root to 2.

    The inverse has all the edge labels and one of the vertex labels. The edge labels will correspond to $n-1$ type D folded duals where they lie in the same orbit under the action of rotating the outer circle.

    Adding the vertex label will correspond uniquely to one of the type D folded duals in the orbit.

    Therefore, type D Goulden-Yong dual gives a bijection between folded duals of a set of minimal factorizations in $D_n$ and $DT_n$.

    By Lemma 4.14, we know that each folded dual corresponds to two minimal factorizations. Therefore, this is a 2-to-1 correspondence between set of minimal factorizations in $D_n$ and type D Goulden-Yong duals.

    Since type D Goulden-Yong duals are bijective to $DT_n \times [n-1]$ by previous Lemma, there's a 2-to-1 correspondence between the set of minimal factorizations in $D_n$ and $DT_n$.
\end{proof}

\section{Applications of generalized Goulden-Yong duals}
In this section, we will explore two applications of the generalized Goulden-Yong duals. The first application is to define generalized Pr\"{u}fer codes which compose well with the generalized Goulden-Yong duals to provide bijective proof for the number of minimal factorizations. The second application is to count the number of signed minimal factorizations by applying the matrix-tree theorem to a weighted graph.

\subsection{Generalized Pr\"{u}fer codes}
In this section, we will define the generalized Pr\"{u}fer codes and show that they give bijections similar to the one between trees and sequences of numbers given by Pr\"{u}fer codes.

\subsubsection{Pr\"{u}fer codes}

First, we will review the Pr\"{u}fer codes of a labeled tree.

\begin{defi}
    Given a labeled tree on $n+1$ vertices, its \textbf{Pr\"{u}fer code} is constructed in the following way:

    -Make the vertex labeled $n+1$ as root

    -Pick the leaf with the smallest label

    -Record the label of the parent of the chosen vertex and remove the vertex from the rooted labeled tree

    -Repeat the process until there's two vertices left in the rooted labeled tree
\end{defi}

\begin{exa}
Below we have a labeled tree on 4 vertices.
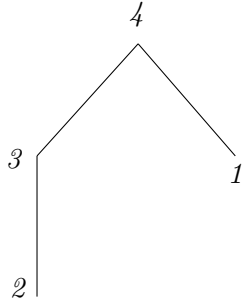
\begin{center}
\begin{tikzpicture}[x=0.75pt,y=0.75pt,yscale=-1,xscale=1]

\draw    (310,93) -- (359,149.6) ;
\draw    (310,93) -- (259,149.6) ;
\draw    (259,149.6) -- (259,220.6) ;

\draw (304,71) node [anchor=north west][inner sep=0.75pt]   [align=left] {4};
\draw (354,152.6) node [anchor=north west][inner sep=0.75pt]   [align=left] {1};
\draw (242,145) node [anchor=north west][inner sep=0.75pt]   [align=left] {3};
\draw (244,210) node [anchor=north west][inner sep=0.75pt]   [align=left] {2};
\end{tikzpicture}    
\end{center}
To construct its Pr\"{u}fer code, we first make $4$ the root and turn it into a rooted tree.

We then find the smallest leaf which is $1$. So we record its parent $4$ in the sequence and remove $1$ from the tree. So $a_1=4$

Then $2$ is the smallest leaf. We record its parent 3 and remove it from the tree. So $a_2=3$.

Now there are two vertices left and we are done.

So its Pr\"{u}fer code is $(4,3)$.
\end{exa}

Define a set of sequences $S_A=\{(a_1,\dots,a_{n-1}) \ | \ a_i \in [n+1]\}$. It's clear that this set has cardinality $(n+1)^{n-1}$. Pr\"{u}fer code gives a bijection between labeled trees of $n+1$ vertices and $S_A$ which also gives a bijective proof for Cayley's formula. 

Composed with the Goulden-Yong duals, this gives a bijective proof for the number of minimal factorizations in $A_n$.

\subsubsection{Type B Pr\"{u}fer codes}

In this section, we will construct the type B Pr\"{u}fer codes.

\begin{defi}
    Given a rooted labeled tree with a loop in $BT_n$, its \textbf{type B Pr\"{u}fer code} $(b_1,\cdots,b_n)$ is constructed in the following way:

    -Record the vertex label of the root as $b_1$.

    -Remove the loop.

    -Let $v_i$ be the unique vertex adjacent to the loop. 
    
    -Pick the smallest leaf.

    -Record the label of the parent of the vertex and remove the vertex from the graph

    -Repeat the process until there's one vertex left in the graph
\end{defi}

\begin{exa}
    Below we have a rooted labeled tree with a loop.

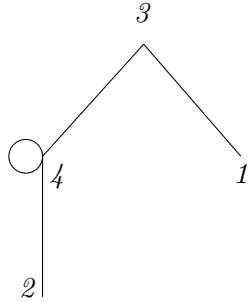
\begin{center}
    \begin{tikzpicture}[x=0.75pt,y=0.75pt,yscale=-1,xscale=1]

\draw    (310,93) -- (359,149.6) ;
\draw    (310,93) -- (259,149.6) ;
\draw    (259,149.6) -- (259,220.6) ;
\draw   (242,149.6) .. controls (242,144.91) and (245.81,141.1) .. (250.5,141.1) .. controls (255.19,141.1) and (259,144.91) .. (259,149.6) .. controls (259,154.29) and (255.19,158.1) .. (250.5,158.1) .. controls (245.81,158.1) and (242,154.29) .. (242,149.6) -- cycle ;

\draw (304,71) node [anchor=north west][inner sep=0.75pt]   [align=left] {3};
\draw (354,152.6) node [anchor=north west][inner sep=0.75pt]   [align=left] {1};
\draw (261,152.6) node [anchor=north west][inner sep=0.75pt]   [align=left] {4};
\draw (246,210) node [anchor=north west][inner sep=0.75pt]   [align=left] {2};
\end{tikzpicture}
\end{center}

To construct its type B Pr\"{u}fer code, we first record the vertex label of the root. So $b_1=3$.

The vertex in the graph that is incident to a loop is $4$. Now make $4$ the new root and remove the loop. We have a rooted tree whose root is $4$.

The smallest leaf is $1$. We record the parent of $1$ and remove $1$, so $b_2=3$.

Then the smallest leaf is 2. Record its parent and remove it, so $b_3=4$.

Now $3$ is the smallest leaf. Record its parent and remove it, so $b_4=4$.

Now we only have one vertex left, we are done. Type $B$ Pr\"{u}fer code for this graph is $(3,3,4,4)$
\end{exa}

\begin{thm}\label{thm F}
    Type $B$ Pr\"{u}fer code gives a bijection between $BT_n$, the set of rooted labeled trees with a loop on n vertices, and a set of sequences $S_B=\{(b_1,\dots,b_n) \ | \ b_i \in [n]\}$
\end{thm}

\begin{proof}
    The sequence $(b_2,\dots,b_n)$ corresponds uniquely to a rooted tree on n vertices by the inverse of the Pr\"{u}fer code.

    Put a loop at the root $b_n$ and make $b_1$ the new root of the graph. This corresponds uniquely to a rooted labeled tree with root $b_1$ with a loop at $b_n$ in $BT_n$.
\end{proof}

Since $(b_2,\dots,b_n)$ gives a rooted tree, we can breakdown $BT_n$ as $[n] \times RT_n$ where $RT_n$ is the set of rooted trees on $n$ vertices.

\begin{cor}
    There's a bijection between $BT_n$ and $[n] \times RT_n$, where $RT_n$ is the set of rooted trees on $n$ vertices.
\end{cor}

\subsubsection{Type D Pr\"{u}fer codes}
In this section, we will construct the type D Pr\"{u}fer codes.

In order to define type D Pr\"{u}fer codes, we first need to modify the type D Goulden-Yong duals into certain kinds of rooted trees.

\begin{lemma}
    If we choose k vertices to be in the cycle of type D Goulden-Yong dual, for any vertex $v_i$ not in the cycle, there's only one possible edge between $v_i$ and one of the vertices in the cycle.
\end{lemma}

\begin{proof}
    If we choose k vertices to be in the cycle of type D Goulden-Yong dual, then the cycle in the folded chord diagram partitions the rest of the diagram into $k-1$ regions.

    $v_i$ in the Goulden-Yong dual corresponds to an edge $e_i$ in the folded chord diagram that's incident to a vertex $u_j$ in the cycle and a vertex $u_k$ outside the cycle. $u_k$ lies in one of the $k-1$ regions and within the region, it's in the clockwise position of an edge in the cycle bounding the region and in the counterclockwise position of another edge in the cycle bounding the region. 
    
    Since the edge label around a vertex is clockwise decreasing, depending on if the edge labels of $e_i$ are bigger or smaller than the edge labels in the cycle, there's a unique vertex in the cycle incident to $e_i$. If the region contains $1$, then $e_i$ is incident to $1$ if the label of $e_i$ is between the biggest and the smallest labels in the cycle.
\end{proof}

\begin{prop}
    There's a $(n-1)$-to-1 correspondence between the set $DT_n$ to the set of rooted trees on $n$ vertices where the root is smaller than all its children.
\end{prop}

\begin{proof}
    If we ignore the root for the graphs in $DT_n$, each unrooted graph has $n-1$ copies in $DT_n$.

    Let $v_i$ be the smallest vertex in the cycle. Remove all edges in the cycle and add edges between $v_i$ and any other vertices in the cycle.

    Make $v_i$ the root of the tree. We now have a rooted tree.

    The root is smaller than all its children because we chose the smallest vertex in the cycle.

    By the previous lemma, each unrooted graph is determined by choosing a set of vertices in the cycle and choosing how the rest of the vertices form forests that are attached to the cycle. So each unrooted graph in $DT_n$ can be uniquely turned into a rooted tree on $n$ vertices where the root is smaller than all its children.
\end{proof}

The following construction is provided by Olivier Bernardi.

\begin{thm}[Bernardi] \label{Bernardi}
    Let $S=\{$rooted trees on $n$ vertices such that the root is smaller than its children$\}$ and 
    $T=\{$rooted trees on $n$ vertices such that $v_n$ is a leaf$\}$. There's a bijection between $g: S \rightarrow T$.
\end{thm}

\begin{proof}
    We'll first describe the map $g$:
    
    Take $s \in S$, let $v_k$ be the root of $s$.

    Take the rooted forest of all descendants of $v_n$ where the roots are the children of $v_n$ and move it under $v_k$, i.e., make descendants of $v_n$ become descendants of $v_k$ while preserve the tree structure.

    There's a unique path $P$ between $v_k$ and $v_n$ in $s$.

    For all children of $v_k$ that are not on the path, we can arrange them in descending order, i.e. there's a sequence $(v_{k_1},\dots,v_{k_j})$ such that $k_i > k_{i+1}$.

    Make $v_{k_{i+1}}$ a child of $v_{k_i}$ for all $1 \leq i \leq j$ while preserving all other descendants.
    
    Make $v_{k_1}$ the root of $s$ instead of $v_k$. This completes the construction of $g(s)$.\\

    Now we will describe $g^{-1}$:

    Take $t \in T$, let $v_k$ be the root of $t$.

    There's a unique path $P$ between $v_k$ and $v_n$ in $t$.

    Order the vertices on the path $P$ as $(v_k,v_{i_1},v_{i_2},\dots,v_n)$ such that $v_{i_j}$ is a child of $v_{i_{j-1}}$.

    There exists an unique element $v_{i_l}$ such that $v_k > v_{i_1} > \dots > v_{i_l}$. (If the whole sequence is ascending, take $v_{i_l}$ to be $v_k$)

    Take the rooted forest of all descendants of $v_{i_l}$ where the roots are the children of $v_{i_l}$ and move it under $v_n$, i.e., make descendants of $v_{i_l}$ become descendants of $v_n$ while preserve the tree structure.

    For elements on the left of $v_{i_l}$ in the sequence, i.e. $(v_k,v_{i_1},\dots,v_{i_{l-1}})$, make each of them a child of $v_{i_l}$ along with all their descendants.

    Now $v_{i_l}$ is the root for the tree.

    This complete the construction of $g^{-1}$.
\end{proof}

\begin{rema}
    T is bijective to the set $\{f \ | \ f:[n-1] \rightarrow [n-1]\}$ based on Andr\'e  Joyal's bijective proof of Cayley's Formula\cite{JOYAL19811} if we view T as a set of vertebrates with the tail being $v_n$.
\end{rema}

\begin{defi}
    Given a graph in $DT_n$, its \textbf{type D Pr\"{u}fer code} $(d_1,\cdots,d_n)$ is constructed in the following way:

    -Record the vertex label of the root as $d_1$.

    -Unroot the graph and construct the corresponding rooted tree in $S$
    
    -Construct the corresponding graph in $T$ via the bijection above. For the steps below, we will work with this graph.

    -Record the vertex label of the parent of $v_n$ as $d_2$ and remove $v_n$ from the graph
    
    -Pick the smallest leaf.

    -Record the label of the parent of the vertex and remove the vertex from the graph

    -Repeat the process until there's one vertex left in the graph
\end{defi}

\begin{exa}
Below we have a rooted graph of $4$ edges on $4$ vertices.
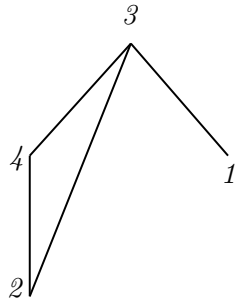
\begin{center}

\tikzset{every picture/.style={line width=0.75pt}} 

\begin{tikzpicture}[x=0.75pt,y=0.75pt,yscale=-1,xscale=1]

\draw    (310,93) -- (359,149.6) ;
\draw    (310,93) -- (259,149.6) ;
\draw    (259,149.6) -- (259,220.6) ;
\draw    (310,93) -- (259,220.6) ;

\draw (304,72) node [anchor=north west][inner sep=0.75pt]   [align=left] {3};
\draw (354,152.6) node [anchor=north west][inner sep=0.75pt]   [align=left] {1};
\draw (247,143.6) node [anchor=north west][inner sep=0.75pt]   [align=left] {4};
\draw (246,210) node [anchor=north west][inner sep=0.75pt]   [align=left] {2};

\end{tikzpicture}
\end{center}

To construct its type D Pr\"{u}fer code, we first record the vertex label of the root. So $d_1=3$.

Then we turn it into a graph in $S$ by removing the edges in the cycle and add one edge between $2$ and $3$ and one edge between $2$ and $4$.

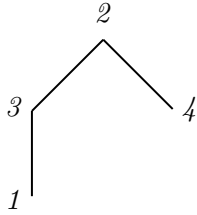
\begin{center}
\tikzset{every picture/.style={line width=0.75pt}} 

\begin{tikzpicture}[x=0.75pt,y=0.75pt,yscale=-1,xscale=1]

\draw    (279,86.6) -- (243,122.6) ;
\draw    (279,86.6) -- (314,121.6) ;
\draw    (243,122.6) -- (243,165.6) ;

\draw (273,67) node [anchor=north west][inner sep=0.75pt]   [align=left] {2};
\draw (228,114) node [anchor=north west][inner sep=0.75pt]   [align=left] {3};
\draw (317,114) node [anchor=north west][inner sep=0.75pt]   [align=left] {4};
\draw (228,161) node [anchor=north west][inner sep=0.75pt]   [align=left] {1};

\end{tikzpicture}
\end{center}

Then we can turn it into a graph in $T$ by making $2$ a child of $3$ since $3$ is not on the path between $2$ and $4$

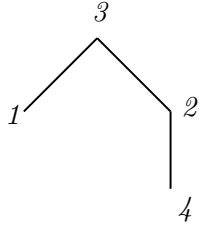
\begin{center}

\tikzset{every picture/.style={line width=0.75pt}} 

\begin{tikzpicture}[x=0.75pt,y=0.75pt,yscale=-1,xscale=1]

\draw    (280,159.6) -- (243,122.6) ;
\draw    (280,159.6) -- (280,198.6) ;
\draw    (243,122.6) -- (206,159.6) ;

\draw (284,151) node [anchor=north west][inner sep=0.75pt]   [align=left] {2};
\draw (239,103) node [anchor=north west][inner sep=0.75pt]   [align=left] {3};
\draw (282,201.6) node [anchor=north west][inner sep=0.75pt]   [align=left] {4};
\draw (195,155) node [anchor=north west][inner sep=0.75pt]   [align=left] {1};

\end{tikzpicture}
\end{center}

Now we record the parent of $4$ and remove $4$. So $d_2=2$.

Then the smallest leaf is 1. So we record its parent and remove $1$. So $d_3=3$.

The smallest leaf now is $2$. So we record its parent and remove $2$. So $d_4=3$.

Therefore, the type D Pr\"{u}fer code for this graph is $(3,2,3,3)$

\end{exa}

\begin{thm}\label{thm G}
    Type $D$ Pr\"{u}fer code gives a bijection between $DT_n$, the set of  rooted simple spanning graphs on $n$ vertices with n edges where the root cannot be the smallest vertex in the cycle, and a set of sequences $S_D=\{(d_1,\dots,d_n) \ | \ d_i \in [n-1]\}$
\end{thm}

\begin{proof}
    $\{(d_3,\dots,d_n)\}$ is bijective to the set of rooted tree on $n-1$ vertices. 

    $d_2$ indicates where to insert $v_n$ in the rooted tree, so $\{(d_2,\dots,d_n)\}$ is bijective to $T$, the set of rooted trees on $n$ vertices such that $v_n$ is a leaf.

    Since there's a $n-1$-to-$1$ correspondence $DT_n$ and $T$ where $d_1$ indicates the root of $DT_n$, then $\{d_1,\dots,d_n\}$ is bijective to $S_D=\{(d_1,\dots,d_n) \ | \ d_i \in [n-1]\}$
\end{proof}

\subsection{Counting type A signed minimal factorizations via matrix-tree theorem}

We will review a version of the matrix-tree theorem for directed multigraphs.

\begin{defi}
    The \textbf{Laplacian of a directed multigraph $G$} with vertices $\{v_1,\dots,v_n\}$ is defined to be the $n \times n$ matrix whose $(i,j)$-th entry $l_{ij}$ is defined to be 

    \begin{equation}
    l_{ij}=\begin{cases}
    \text{outdegree of $v_i$}, & \text{if $i=j$},\\
    \text{-(number of directed edges from $v_i$ to $v_j$)}, & \text{if $i \neq j$}.
    \end{cases}
    \end{equation}
\end{defi}

\begin{thm}[Matrix-tree theorem for directed multigraph]
    Let G be a directed multigraph, L be its Laplacian and $L_i$ be the minor of $L$ with $i$-th row and column removed. Then the number of spanning trees oriented toward $v_i$ is equal to $det(L_i)$.
\end{thm}

For each $A_n$, we can assign a graph to it so that applying the matrix-tree theorem will give us the number of minimal factorizations in $A_n$

\begin{defi}
Define a directed multigraph $G_{A_n}$ as follows: 
$G_{A_n}$ contains $n+1$ vertices. Edges are defined as follows:

-Between $v_{n+1}$ and $v_i$ where $i \in [n]$, there's a directed edge from $v_i$ to $v_{n+1}$ for each $i \in [n]$

-Between $v_{i}$ and $v_j$ where $i\neq j \in [n]$, there's a directed edge from $v_i$ to $v_j$ and a directed edge from $v_j$ to $v_i$
\end{defi}

\begin{prop}
    Denote the Laplacian of $G_{A_n}$ by $L$, then $L_{n+1}$ is the number of minimal factorizations in $A_n$
\end{prop}

\begin{proof}
    $\det(L_{n+1})$ counts the number of spanning trees oriented toward $v_{n+1}$ in $G_{A_n}$ and these are precisely the Goulden-Yong duals in $A_n$.

    Since the Goulden-Yong duals are bijective to the minimal factorizations, therefore $\det(L_{n+1})$ counts the number of minimal factorizations in $A_n$
\end{proof}

To count signed minimal factorizations, we want to give weights to edges in $G_{A_n}$. Before that, we need a result by Igusa and Sen that characterize the relatively projective objects for a special Coxeter element.

\begin{thm}\cite{igusa2023exceptional}
    For a Coxeter element of cycle type $(1,2,\dots,{n+1})$ in $A_n$. 

    A vertex in a Goulden-Yong dual corresponds to a one-way reflection if and only if it's a descent or a child of the root

    A vertex in a Goulden-Yong dual corresponds to a two-way reflection if and only if it's an ascent and not a child of the root.
\end{thm}

In a rooted labeled tree, each vertex except the root can be associated to the edge between the vertex and its parent vertex. If we interpret Goulden-Yong dual as a directed tree oriented towards the root, then an ascent $v_i$ with its parent $v_j$ is associated to the directed edge from $v_i$ to $v_j$ such that $i>j$ and $j \neq n+1$. A descent $v_i$ with its parent $v_j$ is associated to the directed edge from $v_i$ to $v_j$ such that $i<j$ and $j \neq n+1$. Children of the root $v_i$ are associated to the directed edge from $v_i$ to $v_{n+1}$

\begin{defi}
    For the directed edges in the graph $G_{A_n}$, we can define a weight function on the set of directed edges $w:E \rightarrow \{1,2\}$ where $w(e)=1$ if $e$ is a directed edge from $v_i$ to $v_j$ such that $i > j$ and $i\neq n+1$. $w(e)=2$ if $e$ is a directed edge from $v_i$ to $v_j$ such that $i < j$ (including the case $j=n+1$).
    
    Denote Laplacian of $G_{A_n}$ by L. Define a weighted version $L'$ as follows:

    \begin{equation}
    l'_{ij}=\begin{cases}
    \sum_{e_i \text{are directed edges from $v_i$}} w(e_i), & \text{if $i=j$},\\
    0, & \text{if i=n+1} \\
    1, & \text{if $i > j$ and $i \neq n+1$} \\
    2, & \text{if $i < j$} \\
    \end{cases}
    \end{equation}
\end{defi}

\begin{cor}\label{thm H}
    For the weighted Laplacian $L'$ of $G_{A_n}$, $det(L'_{n+1})=n!C_{n+1}$ where $C_n$ is the n-th Catalan number.
\end{cor}

\begin{proof}
    We can see that $det(L'_{n+1})$ counts the number of signed minimal factorizations in $A_n$. The number of signed minimal factorizations in $A_n$ is $n!C_{n+1}$
\end{proof}

\begin{exa} For $A_3$, $G_{A_3}$ is the following directed graph.
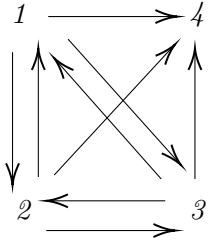
\begin{center}
\begin{tikzpicture}[x=0.75pt,y=0.75pt,yscale=-1,xscale=1]

\draw    (209,107.6) -- (209,172.6) ;
\draw [shift={(209,174.6)}, rotate = 270] [color={rgb, 255:red, 0; green, 0; blue, 0 }  ][line width=0.75]    (10.93,-3.29) .. controls (6.95,-1.4) and (3.31,-0.3) .. (0,0) .. controls (3.31,0.3) and (6.95,1.4) .. (10.93,3.29)   ;
\draw    (227,89.6) -- (290,89.6) ;
\draw [shift={(292,89.6)}, rotate = 180] [color={rgb, 255:red, 0; green, 0; blue, 0 }  ][line width=0.75]    (10.93,-3.29) .. controls (6.95,-1.4) and (3.31,-0.3) .. (0,0) .. controls (3.31,0.3) and (6.95,1.4) .. (10.93,3.29)   ;
\draw    (226,197.6) -- (289,197.6) ;
\draw [shift={(291,197.6)}, rotate = 180] [color={rgb, 255:red, 0; green, 0; blue, 0 }  ][line width=0.75]    (10.93,-3.29) .. controls (6.95,-1.4) and (3.31,-0.3) .. (0,0) .. controls (3.31,0.3) and (6.95,1.4) .. (10.93,3.29)   ;
\draw    (222,170.6) -- (222,107.6) ;
\draw [shift={(222,105.6)}, rotate = 90] [color={rgb, 255:red, 0; green, 0; blue, 0 }  ][line width=0.75]    (10.93,-3.29) .. controls (6.95,-1.4) and (3.31,-0.3) .. (0,0) .. controls (3.31,0.3) and (6.95,1.4) .. (10.93,3.29)   ;
\draw    (301,171.6) -- (301,107.6) ;
\draw [shift={(301,105.6)}, rotate = 90] [color={rgb, 255:red, 0; green, 0; blue, 0 }  ][line width=0.75]    (10.93,-3.29) .. controls (6.95,-1.4) and (3.31,-0.3) .. (0,0) .. controls (3.31,0.3) and (6.95,1.4) .. (10.93,3.29)   ;
\draw    (286,182.6) -- (228,182.6) ;
\draw [shift={(226,182.6)}, rotate = 360] [color={rgb, 255:red, 0; green, 0; blue, 0 }  ][line width=0.75]    (10.93,-3.29) .. controls (6.95,-1.4) and (3.31,-0.3) .. (0,0) .. controls (3.31,0.3) and (6.95,1.4) .. (10.93,3.29)   ;
\draw    (237,101) -- (292.67,163.11) ;
\draw [shift={(294,164.6)}, rotate = 228.13] [color={rgb, 255:red, 0; green, 0; blue, 0 }  ][line width=0.75]    (10.93,-3.29) .. controls (6.95,-1.4) and (3.31,-0.3) .. (0,0) .. controls (3.31,0.3) and (6.95,1.4) .. (10.93,3.29)   ;
\draw    (284,171.6) -- (231.33,112.1) ;
\draw [shift={(230,110.6)}, rotate = 48.48] [color={rgb, 255:red, 0; green, 0; blue, 0 }  ][line width=0.75]    (10.93,-3.29) .. controls (6.95,-1.4) and (3.31,-0.3) .. (0,0) .. controls (3.31,0.3) and (6.95,1.4) .. (10.93,3.29)   ;
\draw    (230,169.6) -- (289.64,105.07) ;
\draw [shift={(291,103.6)}, rotate = 132.75] [color={rgb, 255:red, 0; green, 0; blue, 0 }  ][line width=0.75]    (10.93,-3.29) .. controls (6.95,-1.4) and (3.31,-0.3) .. (0,0) .. controls (3.31,0.3) and (6.95,1.4) .. (10.93,3.29)   ;

\draw (207,82) node [anchor=north west][inner sep=0.75pt]   [align=left] {1};
\draw (209,181) node [anchor=north west][inner sep=0.75pt]   [align=left] {2};
\draw (297,181) node [anchor=north west][inner sep=0.75pt]   [align=left] {3};
\draw (297,81) node [anchor=north west][inner sep=0.75pt]   [align=left] {4};

\end{tikzpicture}
\end{center}

Laplacian of $G_{A_3}$ is $L=
\begin{bmatrix}
3 & -1 & -1 & -1\\
-1 & 3 & -1 & -1\\
-1 & -1 & 3 & -1\\
0 & 0 & 0 & 0\\
\end{bmatrix}$
\end{exa}

$det(L_4)=16=4^2$ which is the number of minimal factorizations in $A_3$.

The weighted Laplacian of $G_{A_3}$ is $L'=
\begin{bmatrix}
6 & -2 & -2 & -2\\
-1 & 5 & -2 & -2\\
-1 & -1 & 4 & -2\\
0 & 0 & 0 & 0\\
\end{bmatrix}$

$det(L_4')=84=3!\cdot 14=3!C_4$ which is the number of signed minimal factorizations in $A_3$.

\begin{rema}
    We can also count the number of minimal factorizations in $B_n$ and $D_n$ using the matrix-tree theorem since $BT_n$ and $DT_n$ can be broken down into copies of rooted trees.

    However, in order to count signed minimal factorizations using weighted Laplacian, one first needs the characterizations of one-way reflections in some Coxeter elements in $B_n$ and $D_n$ which would be the next step for research.
\end{rema}

\subsection*{Acknowledgement} The authors thank Matthieu Josuat-Verg\`es for suggesting to them the equivalence between relative projectivity and one-wayness.
The first author thanks Theo Douvropoulos for helpful conversations about one-way reflections. The authors thank Olivier Bernardi for providing Theorem $\ref{Bernardi}$.



\begin{thebibliography}{10}


\bibitem{apostolakis2018duality}
Nikos Apostolakis.
\newblock A duality for labeled graphs and factorizations with applications to
  graph embeddings and hurwitz enumeration, 2018.

\bibitem{BESSIS2003647}
David Bessis.
\newblock The dual braid monoid.
\newblock {\em Annales Scientifiques de l'\'Ecole Normale Sup\'erieure},
  36(5):647--683, 2003.

\bibitem{biane2019noncrossing}
Philippe Biane and Matthieu Josuat-Verg{\`e}s.
\newblock Noncrossing partitions, bruhat order and the cluster complex.
\newblock In {\em Annales de l'Institut Fourier}, volume~69, pages 2241--2289,
  2019.

\bibitem{MR1265279}
William Crawley-Boevey.
\newblock Exceptional sequences of representations of quivers [{MR}1206935
  (94c:16017)].
\newblock In {\em Representations of algebras ({O}ttawa, {ON}, 1992)},
  volume~14 of {\em CMS Conf. Proc.}, pages 117--124. Amer. Math. Soc.,
  Providence, RI, 1993.

\bibitem{MR1897927}
Ian Goulden and Alexander Yong.
\newblock Tree-like properties of cycle factorizations.
\newblock {\em J. Combin. Theory Ser. A}, 98(1):106--117, 2002.

\bibitem{igusa2023exceptional}
Kiyoshi Igusa and Emre Sen.
\newblock Exceptional sequences and rooted labeled forests.
\newblock Journal of Algebra and Its Applications 28 Mar 2024

\bibitem{igusa2017signed}
Kiyoshi Igusa and Gordana Todorov.
\newblock Signed exceptional sequences and the cluster morphism category'
\newblock Journal of Pure and Applied Algebra 2026, doi.org/10.1016/j.jpaa.2026.108249.

\bibitem{igusa26history}
Kiyoshi Igusa and Gordana Todorov.
\newblock Short history of signed exceptional sequences.
\newblock {Journal of Pure and Applied Algebra}, 230(5), 2026, 108249.

\bibitem{ingalls2009noncrossing}
Colin Ingalls and Hugh Thomas.
\newblock Noncrossing partitions and representations of quivers.
\newblock {\em Compositio Mathematica}, 145.

\bibitem{MR3395490}
Matthieu Josuat-Verg\`es.
\newblock Refined enumeration of noncrossing chains and {H}ook formulas.
\newblock {\em Ann. Comb.}, 19(3):443--460, 2015.

\bibitem{JOYAL19811}
Andr\'e Joyal.
\newblock Une th\'eorie combinatoire des s\'eries formelles.
\newblock {\em Advances in Mathematics}, 42(1):1--82, 1981.

\bibitem{ojakian2019combinatorial}
Kerry Ojakian.
\newblock A combinatorial interpretation of the bijection of goulden and yong.
\newblock {\em AKCE International Journal of Graphs and Combinatorics}, 2019.

\bibitem{reading2007clusters}
Nathan Reading.
\newblock Clusters, coxeter-sortable elements and noncrossing partitions.
\newblock {\em Transactions of the American Mathematical Society},
  359(12):5931--5958, 2007.

\bibitem{ringel1994braid}
Claus~Michael Ringel.
\newblock The braid group action on the set of exceptional sequences of a
  hereditary artin algebra.
\newblock {\em Contemporary Mathematics}, 171:339--339, 1994.

\bibitem{article}
Jian-yi Shi.
\newblock The enumeration of coxeter elements.
\newblock {\em Journal of Algebraic Combinatorics}, 6:161--171, 04 1997.

\end{thebibliography}

\end{document}